\documentclass[a4paper, reqno]{amsart}
\usepackage{fullpage}
\usepackage{hyperref}
\usepackage[all]{xy}
\usepackage[usenames]{xcolor}
\usepackage{amsmath}
\usepackage{amsfonts}
\usepackage{amscd}
\usepackage{amssymb}

\newtheorem{thm}{Theorem}[section]
\newtheorem{prop}[thm]{Proposition}
\newtheorem{lemma}[thm]{Lemma}
\newtheorem{sublemma}[thm]{Sublemma}

\newtheorem{cor}[thm]{Corollary}

\newtheorem{rem}[thm]{Remark}

\newtheorem{notn}[thm]{Notation}

\theoremstyle{definition}
\newtheorem{defn}[thm]{Definition}
\newtheorem{eg}[thm]{Example}

\newcommand{\lm}{\ensuremath{\longrightarrow}}
\newcommand{\Rhom}{\mathbf{R}\strut\kern-.2em\operatorname{Hom}\nolimits}
\newcommand{\RG}{\mathbf{R}\strut\kern-.2em\operatorname{\Gamma}\nolimits}
\newcommand{\D}{\mathsf{D}}
\DeclareMathOperator{\moduleCategory}{\mathsf{mod}} \renewcommand{\mod}{\moduleCategory}
\newcommand{\Amod}{\mod A}

\newcommand{\pd}{\operatorname{pd}\nolimits}
\newcommand{\h}{h}
\renewcommand{\L}{\mathbf{L}}

\newcommand{\xlm}{\ensuremath{\xrightarrow}}

\renewcommand{\O}{\mathcal O}
\newcommand{\x}{\vec{x}}
\renewcommand{\c}{\vec{c}}

\newcommand{\eps}{\varepsilon}
\newcommand{\Gm}{\ensuremath{\mathbb{G}_{\text{m}}}}

\newcommand{\Hom}{\operatorname{Hom}\nolimits}

\DeclareMathOperator{\shom}{\ensuremath{\mathcal{H}\mathit{om}}}

\newcommand{\Ext}{\operatorname{Ext}\nolimits}
\newcommand{\End}{\operatorname{End}\nolimits}
\newcommand{\Aut}{\operatorname{Aut}\nolimits}
\newcommand{\Tor}{\operatorname{Tor}\nolimits}
\DeclareMathOperator{\id}{\rm id}

\DeclareMathOperator{\im}{\mbox{im}\,}
\DeclareMathOperator{\spec}{\mbox{Spec}\,}
\DeclareMathOperator{\proj}{\mbox{Proj}\,}
\DeclareMathOperator{\mo}{\mbox{mod}}

\DeclareMathOperator{\Gr}{\mathsf{Gr}\,}

\DeclareMathOperator{\rk}{\mbox{rank}\,}
\DeclareMathOperator{\coker}{\mbox{coker}\,}

\DeclareMathOperator{\coh}{\mathsf{coh}}
\DeclareMathOperator{\qcoh}{\mathsf{qcoh}}
\DeclareMathOperator{\mmm}{\mathfrak{m}}
\DeclareMathOperator{\s}{\sigma}

\DeclareMathOperator{\G}{G}

\DeclareMathOperator{\w}{\omega}

\DeclareMathOperator{\PP}{\mathbb{P}}

\DeclareMathOperator{\Z}{\mathbb{Z}}
\DeclareMathOperator{\N}{\mathbb{N}}
\DeclareMathOperator{\W}{\mathbb{W}}
\DeclareMathOperator{\X}{\mathbb{X}}
\DeclareMathOperator{\Y}{\mathbb{Y}}
\DeclareMathOperator{\calr}{\mathcal{R}}
\DeclareMathOperator{\calu}{\mathcal{U}}
\DeclareMathOperator{\calm}{\mathcal{M}}
\DeclareMathOperator{\calo}{\mathcal{O}}
\DeclareMathOperator{\calt}{\mathcal{T}}
\DeclareMathOperator{\calot}{\tilde{\mathcal{O}}}
\DeclareMathOperator{\caln}{\mathcal{N}}
\DeclareMathOperator{\call}{\mathcal{L}}
\DeclareMathOperator{\calp}{\mathcal{P}}

\DeclareMathOperator{\cala}{\mathcal{A}}

\DeclareMathOperator{\calc}{\mathcal{C}}

\DeclareMathOperator{\oy}{\mathcal{O}_{Y}}

\renewcommand{\L}{\mathbf{L}}

\begin{document}

\title{Moduli stacks of Serre stable representations in tilting theory}
\author[Chan]{Daniel Chan}
\address{D. Chan: School of Mathematics and Statistics, UNSW, Sydney, 2052, Australia.}
\email{danielc@unsw.edu.au}
\urladdr{http://web.maths.unsw.edu.au/~danielch/}

\author[Lerner]{Boris Lerner}
\address{B. Lerner: School of Mathematics and Statistics, UNSW, Sydney, 2052, Australia.}
\email{boris@unsw.edu.au}
\urladdr{http://www.unsw.edu.au/~borislerner}

\thanks{This project was supported by the Australian Research Council, Discovery Project Grant DP0880143.}

\thanks{Boris Lerner was partly supported by JSPS}

\begin{abstract}
We introduce a new moduli stack, called the Serre stable moduli stack, which corresponds to studying families of point objects in an abelian category with a Serre functor. This allows us  in particular, to  re-interpret the classical derived equivalence between most concealed-canonical algebras and weighted projective lines by showing they are induced by the universal sheaf on the Serre stable moduli stack. We explain why the method works by showing that the Serre stable moduli stack is the tautological moduli problem that allows one to recover certain nice stacks such as weighted projective lines from their moduli of sheaves. As a result, this new stack should be of interest in both representation theory and algebraic geometry. 
\end{abstract}
\maketitle
Throughout, we work over an algebraically closed base field $k$ of characteristic zero.

\section{Introduction}  \label{sintro}  
Tilting theory has proved to be an extremely fruitful avenue of research linking the theory of algebraic geometry to representation theory. In particular, it has produced derived equivalences
between certain classes of projective stacks and certain classes
of finite dimensional algebras. 

The usual way to set up a derived equivalence is to start
with a projective stack $\mathbb{Y}$ and look for a tilting
complex $T^{\bullet} \in \D^b(\mathbb{Y})$. Then $\Y$ will
be derived equivalent to the algebra $A = \End T^{\bullet}$. For example, if $\Y$ is a weighted projective line as defined by Geigle-Lenzing \cite{GL} and $T$ is a tilting bundle, then the endomorphism algebras $A$ are the concealed-canonical algebras of Lenzing-Meltzer \cite{LM}  which include Ringel's canonical algebras \cite{R} as examples. From this perspective, the main 
question is, how to find the tilting complex. The philosophy of Mukai and Bridgeland-King-Reid \cite{BKR} however, is that derived equivalences in algebraic geometry come stereotypically from moduli problems, the equivalence being given by a Fourier-Mukai transform with kernel the dual universal family. The tilting condition is then elegantly explained through orthogonality of members of the universal family. 

From this point of view, it is more natural to start from the other side, in our case, a finite dimensional algebra $A$. This is also the natural starting point for the representation theorist, who may be ``given'' an algebra to study. Now, the main question becomes: Which moduli problem should you pose to obtain a derived equivalent stack? The traditional approach  (see for example \cite{K1}) is to use quiver GIT and works well enough in the case when $A$ is derived equivalent to a projective scheme $\Y$. You choose some
discrete invariant $\vec{d} \in K_0(A)$ (i.e. a dimension vector) and start with the rigidified moduli stack $\X$ of $A$-modules with dimension vector $\vec{d}$ (see section~\ref{sec:rigidmod}). It turns out that the stack $\X$ is naturally represented as the quotient stack $[\calr/PG]$ where $\calr$ is the space of representations (with chosen basis) and $PG$ is the group of basis change modulo scalars. Thus on choosing a stability condition (which roughly corresponds to choosing a nice open substack of $\X$), one can take a GIT quotient of $\calr$ to produce a $\mathbb{G}_m$-quotient stack. 

It is tempting to guess that whenever $A$ is derived equivalent to a projective stack $\Y$ via a tilting bundle, that $\Y$ can be recovered as some open substack
of $\X$ (for some dimension vector and stability condition). However, an elementary computation in the case where $A$ is a canonical algebra other than the Kronecker algebra (and hence derived equivalent to a weighted projective line $\Y$ which is not a scheme) then there is no open substack of $\X$
which is isomorphic to $\Y$. The key problem is that the (rigidified) automorphism groups of modules do not match up with the inertia groups of the derived equivalent stack. In this paper, we introduce a new moduli stack $\X^S$ called the moduli stack of Serre stable representations, which overcomes
these problems for most concealed canonical algebras. Morally speaking, it is described as follows. The shifted Serre or Nakayama functor $\nu_d$ induces a rationally defined self map on $\X$ and $\X^S$ is the fixed point stack of this self map. Alternately, one can motivate this new stack using Bondal-Orlov's \cite{BO} concept of a point object. These objects are {\em Serre stable}, which in the context of finite dimensional $A$-modules $M$ means $M \simeq \nu_d(M)$. From this perspective, $\X^S$ parametrises flat families of Serre stable modules. Now $\nu_d$ induces a linear endomorphism $\Phi$ of $K_0(A)$. If $\X^S$ is to be non-empty, we thus need $\vec{d}$ to be fixed by $\Phi$, in which case we say it is {\em Coxeter stable}. The Serre stability condition also arises naturally in Bridgeland-King-Reid's criterion for an exact functor to be an equivalence \cite[Theorem~2.4]{BKR}. 

The correct setting for our results are smoothly weighted projective varieties (defined Section~\ref{sec:taut}), a notion which generalises weighted projective lines. Essentially, these are stacks which are generically varieties and stacky behaviour is confined to smooth non-intersecting divisors. The first result concerns the Fano or anti-Fano case (defined Section~\ref{sec:universal}), which for concealed canonical algebras corresponds to the non-tubular case. 

\begin{thm}[Theorem~\ref{thm:universal}] \label{thm:1}
Let $\Y$ be a smoothly weighted projective variety which is either Fano or anti-Fano  and $\mathcal{T}$ be a tilting bundle on $\Y$ with non-isomorphic indecomposable summands. If $\X^S$ is the Serre stable moduli stack of representations for the endomorphism algebra $A = \End_{\Y} \mathcal{T}$ corresponding to the dimension vector $\vec{d}$ of $\mathcal{T}$, then $\Y \simeq \X^S$ and the dual $\mathcal{T}^{\vee}$ is the universal representation.
\end{thm}


To understand why such a result should hold, we first note that the Serre stable moduli stack can be defined in fairly general contexts, essentially whenever one works in an abelian category with a
functor. In particular, one can start with a stack $\Y$ and ask if there is some tautological moduli problem in $\coh
\Y$ whose solution is $\Y$ itself. In general, this should not be possible as there are non-isomorphic stacks with isomorphic categories of
coherent sheaves. However, if $\Y$ is a projective scheme, then we can look at the rigidified moduli stack $\W$ of skyscraper sheaves which in
this case, coincides with the Serre stable moduli stack $\W^S$, and this of course recovers $\Y$. When $\Y$ is a smoothly weighted projective variety, then $\W$ and $\W^S$ are no longer isomorphic and it is $\W^S$
which recovers the original stack $\Y$. 
\begin{thm}[Theorem \ref{thm:order}] \label{thm:2}
	Let $\Y$ be a smoothly weighted projective variety and $\W^S$ be the Serre stable moduli stack of ``skyscraper'' sheaves on $\Y$. Then $\W^S\simeq \Y$
\end{thm}
\noindent 
To a large extent, this explains why the Serre stable moduli stack is the correct stack to look at when we have a concealed canonical algebra. Indeed, we use Theorem~\ref{thm:2} to prove Theorem~\ref{thm:1}.

Theorem~\ref{thm:1} suggests a first approach to answering the question: given a finite dimensional algebra, how do you find a derived equivalent stack? You pick a Coxeter stable dimension vector, compute the Serre stable moduli stack and then check if the dual of the universal representation is tilting. However, one might hope for more. Indeed in \cite{BKR}, the tilting condition comes out of the theory and there is no need to check it case by case. Emulating this, we seek module-theoretic criteria for the dual universal sheaf to be tilting. This has the potential for answering questions such as: given a class $\calc$ of stacks, characterise the endomorphism algebras of tilting bundles on objects of $\calc$.  For example, we have the following characterisation of non-tubular concealed canonical algebras. 

\begin{thm}[Theorem~\ref{tBKR} + Remark to Theorem~\ref{thm:universal}]  \label{thm:3}
Let $A$ be a basic connected finite dimensional algebra of finite global dimension and $\vec{d}\in K_0(A)$ be a minimal Coxeter stable dimension vector. Suppose that
\begin{enumerate}
\item the Serre stable moduli stack $\X^S$ is a weighted projective curve, and
\item any Serre stable module $M$ of dimension vector $\vec{d}$ is the direct sum of modules $M_i$ such that every proper submodule $N$ of $M_i$ satisfies $\sum_{i\geq 0} (-1)^i\dim \Ext^i_A(M,N) <0$.
\end{enumerate}
Then the dual of the universal representation is a tilting bundle giving a derived equivalence between $A$ and $\X^S$. In particular, a basic finite dimensional algebra $A$ is non-tubular concealed canonical if and only if it satisfies the hypotheses above and furthermore, $\ker (\Phi - \id_{K_0(A)})$ is 1-dimensional, where $\Phi$ is the Coxeter transformation. 
\end{thm}
\noindent
Lenzing-de la Pe\~{n}a \cite{LdP} characterise concealed-canonical algebras as those with a sincere separating exact subcategory and it would be interesting to see if this theorem can be used to give an elegant alternate proof of this. Note that the proof in \cite{LdP} does not ``produce'' the derived equivalent stack directly, but instead operates very much in the reverse perspective mentioned above: one first computes all the categories of coherent sheaves on weighted projective lines, and then tilts algebras to these. With the aim of independently re-deriving results such as Lenzing-de la Pe\~{n}a's, we will as much as possible, avoid assuming other results which give existence of tilting bundles.

Hypothesis (ii) of Theorem~\ref{thm:3} is a module-theoretic condition related to classical stability. Lenzing and de la Pe\~{n}a essentially show (Proposition~\ref{p:regssforcc}) it holds for algebras $A$ with a sincere separating exact subcategory, en route to establishing their module-theoretic criterion for being concealed-canonical. If furthermore $A$ is not tubular, the minimal Coxeter stable dimension vector is unique so there is no guesswork involved in choosing a dimension vector here. The moduli-theoretic condition (i) is unfortunate, and current work aims to replace it with a module-theoretic condition as occurs in the usual Bridgeland-King-Reid theory. We are forced to include it as  we don't have the required stack technology. The main obstruction is that we don't have a stable reduction theorem to guarantee that $\X^S$ is proper. 

Nevertheless, we do show (from first principles) that the hypotheses of Theorem~\ref{thm:3} hold for all canonical algebras. In this case, a similar result has been reached by Abdelghadir-Ueda \cite{AU} using quiver GIT. However, they consider an ad hoc moduli space of enriched quiver representations instead of our Serre stable moduli stack. Our approach also does not require the choice of a separate stability condition. The Serre stability condition in this case, is enough to remove modules which would otherwise cause the moduli stack to be badly behaved e.g. non-separated. 

We hope the contents of this paper reinforces the following not so well advertised theme in non-commutative algebraic geometry: moduli spaces are an interesting and fruitful way to study non-commutative algebras. We see this theme already appearing in Artin-Tate-Van den Bergh's paper \cite{ATV} which kicked off the study of non-commutative projective geometry by looking at moduli spaces of point modules to unlock secrets in the Sklyanin algebra. In general, given any moduli stack $\mathbb{M}$ of $A$-modules, the universal sheaf $\mathcal{U}$ can be considered an $(\calo_{\mathbb{M}},A)$-bimodule and $\Hom, \otimes$ can be used to relate the categories of quasi-coherent sheaves on $\mathbb{M}$ and $A$-modules. The question is which moduli stacks will easily give interesting information and the point of this paper, is to see how the Serre stable moduli stack is a good candidate in many contexts. 

The outline of this paper is as follows. Section~\ref{sec:rigidmod} reviews aspects of stack theory as it relates to the moduli of $A$-modules. It is aimed at representation theorists. In Section~\ref{sec:serrestable}, we introduce the Serre stable moduli stack. To understand this stack, it is instructive to study its $k$-points, something we do in Section~\ref{sec:kpoints}. Condition~(ii) of Theorem~\ref{thm:3} naturally arises here. In Section~\ref{scanonical}, we study canonical algebras ``afresh'' via moduli spaces. In particular, we compute from first principles the Serre stable moduli stack for canonical algebras and dimension vector $\vec{d} = \mathbf{1}$ and show that it satisfies the hypotheses of Theorem~\ref{thm:3}. In Section~\ref{sec:Beil}, we also compute the Serre stable moduli stack for the Beilinson algebra and dimension vector $\mathbf{1}$, comparing our result with the traditional approach via quiver GIT. In both these examples, we'll observe a nice feature of the Serre moduli stack, that we do not need to choose a separate stability condition as occurs for quiver GIT, and that the choice of dimension vector is essentially locked in. We review cyclic quotient stacks in Section~\ref{sec:serrestacks} and the ``Serre'' functor in this case, in preparation for studying the Serre stable moduli stack of ``skyscraper'' sheaves. The local computations in this section will also clarify the Serre stability condition. Section~\ref{sTistilting} is devoted to proving Theorem~\ref{thm:3} and hence, that the dual of the universal sheaf is tilting in the canonical algebra case. This reproves Geigle-Lenzing's derived equivalence. Theorems~\ref{thm:2} and \ref{thm:1} are then proved in Sections~\ref{sec:taut} and \ref{sec:universal} respectively. 

\vspace{3mm}

\noindent
\textbf{Conventions} Throughout this paper, $A$ will denote a finite dimensional $k$-algebra. Stacks will be denoted using the blackboard bold font such as $\Y, \X, \W$.  
By default, $A$-modules will be right modules, though occasionally, we will need to look at left modules, for example when looking at duals of these modules. Similarly, modules $\calm$ over $\calo_{\Y} \otimes_k A$ will usually be viewed as $(\calo_{\Y},A)$-bimodules with $A$ acting on the right and ``functions'' in $\calo_{\Y}$ acting on the left. The unadorned tensor symbol $\otimes$ will denote the tensor product over $k$. 

\vspace{3mm}
\noindent
\textbf{Acknowledgements} We would like to thank Jack Hall for a very helpful discussion during his visit.
	Furthermore, Boris
	would like to thank Osamu Iyama, Laurent Demonet and Gustavo Jasso for numerous insightful
comments during his stay in Nagoya.

\section{The rigidified moduli stack of modules}\label{sec:rigidmod}

In this section, we recall the moduli stack of modules and its description as a quotient stack. We also recall the less well-known procedure of rigidification. This overview is aimed at representation theorists with only a passing acquaintance with stacks. Although it is too brief to allow such readers to follow all the proofs in this paper, it should allow them to understand and appreciate the results. The reader who wishes to see more details about stacks should consult standard texts such as \cite{LM-B}, \cite{Kr}, \cite{Stacks}. 

Let $A$ be a basic finite dimensional $k$-algebra so we may write $A = kQ/I$ for some quiver $Q=(Q_0,Q_1)$ and admissible ideal $I$ (see \cite{ASS} page 53). We let $e_v \in A$ be the idempotent corresponding to the vertex $v \in Q_0$. We fix a dimension vector $\vec{d}\colon Q_0 \to \N\colon v \mapsto d_v$ which can also be viewed as an element of the Grothendieck group $K_0(A)$. 
Consider the affine space  $\mathbb{A}^Q:=\prod\limits_{v \to w\in Q_1} \Hom_k(k^{d_v},k^{d_w})$ whose $k$-points correspond to the representations of $Q$ of the form $V = k^{\vec{d}} := \oplus_{v} k^{d_v}$ and hence isomorphism classes of $kQ$-modules $M$ with a chosen (ordered) basis (for each $Me_v$). The ideal $I$ determines a closed subscheme $\mathcal{R}$ of $\mathbb{A}^Q$ corresponding to the $A$-modules. Now $\mathcal{R}$ is a fine moduli space parametrising $A$-modules of dimension vector $\vec{d}$ with a chosen basis. Informally, this means that for any commutative ring $R$, the $R$-points of $\mathcal{R}$ correspond to isomorphism classes of $(R \otimes A)$-modules with given dimension vector and $R$-basis. More precisely, let $\mathcal{U} = \calo_{\mathbb{A}^Q}^{\vec{d}}$ be the universal representation of $Q$ on $\mathbb{A}^Q$. The fibre above $p \in \mathbb{A}^Q$ is simply the representation $\calo_p \otimes_{\mathbb{A}^Q} \mathcal{U} = V_p$ corresponding to $p$. Then $\mathcal{U}|_{\mathcal{R}}$ is an $(\calo_{\mathcal{R}} \otimes A)$-module and for any $R$-point $f: \spec{R} \lm \mathcal{R}$, $f^* \mathcal{U}|_{\mathcal{R}}$ is an $(R \otimes A)$-module of dimension vector $\vec{d}$ with a chosen ordered basis. Furthermore, the isomorphism classes of such modules with chosen basis are given by a unique $R$-point in this fashion. We call $ \mathcal{U}|_{\mathcal{R}}$ the {\it universal representation} because of this universal property. 

Unfortunately, if we try to parametrise isomorphism classes of $A$-modules (without chosen basis), we find that there is in general no such universal $A$-module and one main obstruction is that $A$-modules have automorphisms (see \cite[Chapter~2, Section~A]{HM} for an explanation of this phenomenon). Algebraic geometers can often circumvent this obstacle by enlarging the category of schemes to stacks. 

\subsection{Review of stacks}\label{subsec:reviewstacks}

For us, we will view a stack $\mathbb{X}$ as a sheaf of groupoids. The data involved in defining such a stack thus consists of:
\begin{enumerate}
\item for each noetherian test scheme $T$, a groupoid $\mathbb{X}(T)$ viewed as a category, all of whose morphisms are isomorphisms, and
\item for each map of test schemes $f:T' \lm T$, a {\it pullback} functor $f^*:\mathbb{X}(T) \lm \mathbb{X}(T')$.
\end{enumerate}
We omit the long list of axioms these data must satisfy. Informally, the isomorphism classes in $\mathbb{X}(T)$ should be thought of as morphisms $T \lm \mathbb{X}$. Categorifying the set of $T$-points allows us to remember automorphisms which prevented the existence of universal families. Hence we will refer to $\mathbb{X}(T)$ as the {\it category of $T$-points of $\mathbb{X}$}. A morphism of stacks is simply a functor which respects pullback. Any (quasi-separated) scheme $X$ gives rise to a stack (also denoted $X$) defined as follows: $X(T)$ is the category whose objects are the $T$-points of $X$, and the only morphisms are the identity. In this way, the category of (quasi-separated) schemes embeds in the category of stacks. One might wonder if the image of this embedding is the stacks whose category of $T$-points (for all $T$) only have identity morphisms (and so are essentially sets). This is almost true (one needs to include algebraic spaces). 

We now describe the moduli stack $\tilde{\X}$ of $A$-modules (of dimension vector $\vec{d}$). 
Following \cite{K1} we first make the following definition of a family
of $A$-modules:
\begin{defn}
		Let $T$ be scheme. A \emph{flat family of $A$-modules
		over $T$} is a finitely generated $\calo_T\otimes A$-module $\mathcal{M}$ which is locally free over $T$. If
		$T=\spec R$ then we simply say \emph{a flat family over $R$}. If $A$ is described by a quiver with relations then we also call a flat family of modules a \emph{flat family of representations}. This is just a representation of
$Q$ with the given relations in the category of locally free sheaves over $T$. The dimension vector of $\mathcal{M}$ is given by $d_v = \text{rank}_T \mathcal{M}e_v$. 
\end{defn}

Let $\tilde{\mathbb{X}}$ be the stack defined as follows. For a test scheme $T$, objects of $\tilde{\mathbb{X}}(T)$ are the flat families of $A$-modules over $T$ of dimension vector $\vec{d}$ and the morphisms are the isomorphisms in the category of $\calo_T \otimes A$-modules. Given a morphism $f:T' \lm T$ and $\mathcal{M} \in \tilde{\mathbb{X}}(T)$, we have the usual pullback of sheaves which defines $f^*\mathcal{M} \in \tilde{\mathbb{X}}(T')$. This gives the pullback functor of the stack. To show these data do indeed satisfy all the stack axioms, it is usual to identify it with another well-known stack.

For this, we need to introduce the quotient stack construction, which will be vital for us. Let $X$ be a quasi-separated scheme and $G$ an algebraic group acting on $X$. We will need to use the notion of a {\em $G$-torsor} (also called a $G$-bundle or principal homogeneous space for $G$) whose definition can be found in \cite[Chapter~III, \textsection 4]{Mil}. We define the {\it quotient stack} $[X/G]$ whose category of $T$-points consists of diagrams 
\[\xymatrix{
	\tilde{T} \ar[r]^{\tilde{f}} \ar[d]^{\pi} & X \\
	T & 
}\]
where $\tilde{T} \lm T$ is a $G$-torsor and $\tilde{f}: \tilde{T} \lm X$ is a $G$-equivariant morphism. The morphisms are precisely the isomorphisms of this diagram compatible with given structure. There is a natural quotient morphism $X \lm [X/G]$ defined by the functor $X(T) \lm [X/G](T)$ which sends the $T$-point $f:T \lm X$ to the trivial $G$-torsor $G \times T \lm T$ and $G$-equivariant map 
$$ G \times T \xrightarrow{1_G \times f} G \times X \xrightarrow{\alpha} X $$
where $\alpha$ is the action of $G$ on $X$. For quotient stacks, automorphism groups are easily computed. If $X$ is a $k$-variety and $p\in X$ a $k$-point, the automorphism group of its image in $[X/G]$ is simply the stabliser group $\text{Stab}_G p$. For this reason, we will usually refer to these automorphism groups as {\em inertia groups}. 

The stack $\tilde{\mathbb{X}}$ of $A$-modules of dimension vector $\vec{d}$ turns out as one would want, to be the quotient stack of $\mathcal{R}$ by the group of change of bases. More precisely, let  $G = \prod\limits_{v \in
Q_0} GL_{d_v}$ which acts naturally on $\mathcal{R}$. We briefly describe the isomorphism $\tilde{\mathbb{X}} \simeq [\mathcal{R}/G]$. Given a flat family of $A$-modules $\mathcal{M}  = \oplus_{v \in Q_0} \mathcal{M}_v \in \tilde{\mathbb{X}}(T)$, the frame bundle $\pi_v:\tilde{T}_v \lm T$ of the rank $d_v$ vector bundle $\mathcal{M}_v$ is a $GL_{d_v}$-torsor whose fibre above a $k$-point $p$ is just the group of vector space isomorphisms $k^{d_v} \lm \calo_p \otimes_T \mathcal{M}_v$. The fibre product of these frame bundles $\tilde{T}_v, v \in Q_0$ over $T$ gives a $G$-torsor $\pi:\tilde{T} \lm T$ whose $k$-points are just the $A$-modules of form $\calo_p \otimes_T \mathcal{M}$ together with a choice of basis. There is hence a $G$-equivariant morphism $\tilde{f}: \tilde{T} \lm \mathcal{R}$ and the pair $(\pi,\tilde{f})$ defines an element of $[\mathcal{R}/G](T)$. This turns out to define the isomorphism $\tilde{\mathbb{X}} \lm [\mathcal{R}/G]$. 

The inverse isomorphism is given as follows. Consider a $G$-torsor $\pi:\tilde{T} \lm T$ and $G$-equivariant morphism $\tilde{f}: \tilde{T} \lm \mathcal{R}$ defining an object in $[\mathcal{R}/G](T)$. Note first that the universal sheaf $\mathcal{U}|_{\mathcal{R}}$ is naturally a $G$-equivariant sheaf so $f^*  \mathcal{U}|_{\mathcal{R}}$ is a $G$-equivariant sheaf on $\tilde{T}$. This descends (via descent along a torsor) to a flat family of $A$-modules $(f^*\mathcal{U}|_{\mathcal{R}})^G$ over $T$. This defines a functor $[\mathcal{R}/G](T) \lm \tilde{\mathbb{X}}(T)$ and yields the inverse functor. For this reason, we will refer to $\mathcal{U}|_{\mathcal{R}}$ together with its $G$-action as the {\it universal sheaf} on $[\mathcal{R}/G]$.

\subsection{Weighted projective curves}\label{subsec:wtdlines}

The weighted projective line as studied by representation theorists is usually viewed as the quotient stack of a punctured surface by a 1-dimensional group. However, it is more geometrically meaningful to define it by gluing together quotient stacks of the form $[U/\mu_p]$ where $U$ is a 1-dimensional variety and $\mu_p$ is the cyclic group of $p$-th roots of unity. We will use this latter formulation. Furthermore, we will define weighted projective curves since this involves no more work. The relation between the two approaches is made explicit in the appendix.

We start with a smooth projective curve $C$. 
Let $q_1, \ldots, q_n \in C$ be distinct points where we ``weight'' the curve $C$, that is introduce stacky behaviour. Let $p_1, \ldots, p_n$ be integers $\geq 2$ called the {\it weights}. To these data, we define the weighted projective curve $\mathbb{Y} = \mathbb{Y}(\sum p_i q_i)$ together with a morphism $\psi: \mathbb{Y} \lm C$ as follows. Above $U_0 := C - \{p_1,\ldots, p_n\}$, $\psi$ is an isomorphism so $\mathbb{Y}$ has an open substack which is a scheme and in fact, an open subset of the original curve $C$. Above the point $q_i$, we pick a sufficiently small open neighbourhood $U_i \subset C$ disjoint from all the other $q_j$'s. Let $t \in \calo_{U_i}$ be a local parameter defining $q_i$ and $\tilde{U}_i =\spec \calo_{U_i}[s]/(s^{p_i} - t)$ so $\tilde{U}_i \lm U_i$ is a $\mu_{p_i}$-cover of $U_i$ which is totally ramified above $q_i$ and unramified elsewhere. We define $\psi^{-1}(U_i) = [\tilde{U}_i/\mu_{p_i}]$. Note that $\tilde{U}_i \lm U_i$ factors as 
$$ \tilde{U}_i \xrightarrow{\rho} [\tilde{U}_i / \mu_{p_i}] \xrightarrow{\psi_i} U_i$$
where $\rho$ is the quotient map described in Subsection~\ref{subsec:reviewstacks}. To define $\psi_i$, consider an object of $[\tilde{U}_i / \mu_{p_i}](T)$ given by the $G$-torsor $\pi: \tilde{T} \lm T$ and $G$-equivariant map $\tilde{T} \lm U_i$. Then $\psi_i$ of this object is the unique morphism $\beta \in U_i(T) = \Hom(T,U_i)$ which makes the diagram below commute
$$\begin{CD}
   \tilde{T} @>>> \tilde{U_i} \\
    @V{\pi}VV @VVV \\
   T @>{\beta}>> U_i
  \end{CD}$$
 Since $\mu_{p_i}$ acts freely away from ramification, the map $\psi_i: \psi^{-1}(U_i) \lm U_i$ is an isomorphism away from $q_i$. However, the inertia group of the point above $q_i$ is the stabiliser group $\mu_{p_i}$. Hence, in the weighted projective curve $\Y$, the point $q_i$ is replaced with a ``stacky'' point. We call $C$ the {\em coarse moduli scheme} of $\Y$ since it is, in a sense that can be made precise, the ``best'' scheme approximation to $\Y$. When $C = \PP^1$, we call $\Y$ a {\em weighted projective line}. 

We recall that if $T$ is a tilting bundle on $\mathbb{Y}$ and $A$ is the concealed-canonical algebra $\End_{\mathbb{Y}} T$, then the non-stacky $k$-points of $\mathbb{Y}$ correspond to the tubes of $A$ of period 1 and a stacky point with inertia group $\mu_{p_i}$ corresponds to a tube of period $p_i$. 

The inertia groups of the weighted projective line are generically trivial. Every non-zero $A$-module has at least a copy of $\mathbb{G}_m$ in its automorphism group, so the moduli stack $\tilde{\mathbb{X}}$ above is never a weighted projective line. There is however, an easy way to remove this common $\mathbb{G}_m$ from the inertia group which we describe in the next subsection.

\subsection{Rigidification}\label{subsec:rigid}

We describe here the process of rigidification in our specialised context, as one might find for example in \cite[Section~5]{ACV}. One manifestation of the common copy of $\mathbb{G}_m$ in the automorphism groups of $A$-modules is that the diagonal copy of $\mathbb{G}_m$ in $G$ acts trivially on $\calr$. Thus the easy way to remove the common copy of $\mathbb{G}_m$ is to replace $[\calr/G]$ with $[\mathcal{R}/(PG)]$ where $PG = G/\mathbb{G}_m$. 

We now define a moduli stack $\X$ which gives a module-theoretic interpretation of this new quotient stack. We start by defining a pre-stack $\X^{pre}$ (pre-stacks are defined by the same data \textsection~\ref{subsec:reviewstacks} as a stack, but satisfy less axioms). 
For a noetherian test scheme $T$, we let the objects of $\X^{pre}(T)$ be the objects of $\tilde{\X}(T)$. However, given objects $\calm, \caln \in \X^{pre}(T)$ the morphisms from $\calm$ to $\caln$ in $\X^{pre}(T)$ will consist of equivalence classes of isomorphisms
$\theta\colon\calm \to \call \otimes_T \caln$ of $ \calo_T \otimes A$-modules for some line bundle $\call$
on $T$. If $\theta'\colon\calm \to \call' \otimes_T \caln$ is another such morphism, we say $\theta,\theta'$ 
are {\em equivalent} if and only if there is an isomorphism $\ell\colon \call \to \call'$ such that
$\theta' = (\ell \otimes {\rm id}) \theta$. Note that if $\call = \call'$ then this just means
isomorphisms differ by an element of $\mathbb{G}_{m,T}$. Given another isomorphism $\phi\colon\caln
\to \call'' \otimes \calp$ with $\calp \in \X^{pre}(T)$ and $\call''$ a line bundle on $T$, we obtain
another isomorphism $(\call \otimes \phi) \theta\colon \calm \to \call \otimes_T \caln \to 
\call \otimes_T \call'' \otimes_T \calp$. The equivalence class of this isomorphism remains unchanged if
we replace $\theta$ and $\phi$ with equivalent morphisms, so we obtain a composition law
on $\X^{pre}(T)$. Unfortunately, $\X^{pre}$ may fail the sheaf axiom for a stack, but sheafifying or ``stackifying'' it will yield a stack $\X$. 

It will be instructive to look at the special case where there is a vertex $v_0$ such that $d_{v_0} = 1$. This means that if $\calm = \oplus \calm_v \in \X^{pre}(T)$, then $\calm_{v_0}$ is a line bundle on $T$ so we may, up to isomorphism, replace $\calm$ with $\calm_{v_0}^{-1} \otimes_T \calm$ and so assume $\calm_{v_0} \simeq \calo_T$. Performing this replacement is called rigidification which is why $\X$ is called the {\it rigidification} of $\tilde{\X}$. When we introduce the moduli stack of Serre stable representations, we will see it will be convenient to include families which are not necessarily rigidified. There will be many interesting cases where the $d_{v_0} = 1$ assumption holds. For example, if $A$ is canonical, the dimension vector of interest for us will be $\vec{d} = \mathbf{1}$ where all entries are 1, whilst if $A$ is tame hereditary, the dimension vector $\vec{d}$ of interest for us will be the purely imaginary root, which also always has at least one 1 associated to a vertex. 

\begin{prop}
	With the above definitions, we have $\X \simeq [\mathcal{R}/PG]$.
\end{prop}
\begin{proof} (Sketch) We will carry out the proof under the assumption that $v_0 \in Q_0$ is a vertex with $d_{v_0} = 1$ and indicate afterwards how to modify the proof in general. 
	We use \'{e}tale cohomology below. Consider the
	exact sequence of group schemes
	\[ 1 \lm \mathbb{G}_m \lm G \lm PG \lm 1. \] 
This sequence is split since we may project $G$ onto $GL_{d_{v_0}} \simeq \Gm$ and thus identify $PG$ with the subgroup $G' = \{ (g_v)_{v \in Q_0}| g_{v_0} = 1\}$. Hence in the exact sequence below
\[ 0\lm H^1(T,\mathbb{G}_m) \lm H^1(T,G) \xrightarrow{\psi} H^1(T,PG) \xrightarrow{\beta} H^2(T,\mathbb{G}_m)\]
$\psi$ is a split surjection. 

An object of $[\mathcal{R}/PG](T)$ consists of a $PG$-torsor $\tilde{T} \lm T$ and a $PG$-equivariant  map $\phi\colon\tilde T\to \calr$. Now $H^1(T,\mathbb{G}_m), H^1(T,G)$ and $H^1(T,PG)$ classify $\mathbb{G}_m$-torsors, $G$-torsors and $PG$-torsors over $T$ respectively. Hence the split sequence above shows that $\tilde{T}$ comes from a $G$-torsor $T' \lm T$ in the sense that $T'/\mathbb{G}_m = \tilde{T}$. This $G$-torsor is unique up to a $\mathbb{G}_m$-torsor and, together with the $G$-equivariant map $T' \lm \tilde{T} \lm \calr$ defines an object of $\calm \in \tilde{\X}(T)$ by the isomorphism $\tilde{\X} \simeq [\calr/G]$. (In fact, our choice of splitting ensures that $\calm$ is rigidified). From the above discussion, we see that the morphisms in $[\calr/PG](T)$ and $\X^{pre}(T)$ correspond so there is a fully faithful functor $[\calr/PG](T) \lm \X^{pre}(T)$. It is dense since any object $\calm \in \X^{pre}(T)$ can be assumed to be rigidified and hence the frame bundle of $\calm_{v_0}$ is trivial. The corresponding $G$-torsor thus comes from a $G'$-torsor and we are done in this special case. Moreover, $\X^{pre}$ is already a stack and we do not need to stackify it, that is, $\X^{pre} = \X$. 

If we do not make our assumption on our special vertex, then $\psi$ may not be split surjective so a $PG$-torsor, say corresponding to  $\gamma \in H^1(T,PG)$ may not lift to a $G$-torsor. However, it will lift on passing to an \'{e}tale extension as can be seen by splitting the Brauer class $\beta(\gamma)$. 
%
\end{proof}
\noindent
We call $\X$ the {\em rigidified moduli stack of $A$- modules} with dimension vector $\vec{d}$. If the dimension vector needs to be noted, we will denote this stack by $\X_{\vec{d}}$. 

Given an $A$-module $M$ of dimension vector $\vec{d}$ which is a {\it brick} in the sense that $\End_A M = k$, the inertia group of the corresponding point of $\X$ is now trivial as desired. However, automorphism groups of regular $A$-modules like
\[\xymatrix{
	& k \ar[r]^{0} & k  \ar[dr]^{0} & \\
	k \ar[ur]^{0} \ar[rrr]^{1} & & & k 
}\]
are typically products of copies of $\mathbb{G}_m$ so in general, the rigidified moduli stack $\X$ is still not a weighted projective line. In the next section, we introduce the Serre stable moduli stack which rectifies this problem. 

\section{The Serre stable moduli stack $\X^S$}\label{sec:serrestable}

In this section, we introduce the main new object of study, the \emph{moduli stack of Serre stable
representations}. Morally speaking, the Serre functor induces a rational self map on the rigidified moduli stack $\X$ of $A$-modules of dimension vector $\vec{d}$, and we look at the fixed point stack of this map.

Let $A$ be a finite dimensional $k$-algebra with finite global dimension. The \emph{Nakayama
functor}
\[\nu:=-\otimes^\L_ADA,\quad{\rm where}\quad D=\Hom_k(-,k)\]
is the Serre
functor in $\D^b(\Amod)$ in the sense that for any $M,N\in \mod A$ we have 
\[\Hom_{\D^b}(M,N)\simeq D\Hom_{\D^b}(N,\nu M)\]
(see \cite{H}). Furthermore, for any $d\in\Z$ we let $\nu_d:=\nu\circ [-d]$ where $[-d]$ denotes the shift in the triangulated category $\D^b(\Amod)$.

Note that the action of $\nu$ and $\nu_d$ extend to $\D^b(\mod \calo_T\otimes_k A)$, however $\nu$ is no longer a Serre
functor in this category.

We will be interested in {\em Serre stable} $A$-modules i.e. those 
$A$-modules $M$ which satisfy $\nu_d(M)\simeq M$. To show families of such
modules are stable under base change, we will need the following:

\begin{lemma} \label{lserrebchange} 
Let $\calm$ be flat family of
	$A$-modules over $R$ and $R\to S$ be a morphism of commutative noetherian rings. Suppose
	further that the $R$-modules $h_i(\calm\otimes^\L_ADA)$ are flat for $i< n$. Then for all
	$i \leq n$, there is a natural isomorphism 
	\[ S \otimes_R h_i(\calm\otimes^\L_ADA)\simeq h_i(S\otimes_R\calm\otimes^\L_ADA).\]
\end{lemma}
\begin{proof}
From the hypertor homology spectral sequence \cite[Application~5.7.8]{W} with $E^2$ page: 
\[E^2_{ij} = \Tor^R_i(M, h_{j}(K_\bullet))\Rightarrow h_{i+j}(M\otimes^\L_RK_\bullet)\]
where $M$ is an $R$-module and $K_\bullet$ is a bounded below chain complex of free $R$-modules, we have that for
$i\leq n$
\[S\otimes_Rh_i(\calm\otimes^\L_ADA)\simeq
h_i(S\otimes^\L_R \calm\otimes^\L_ADA)
=h_i(S\otimes_R\calm\otimes^\L_ADA)\]
\end{proof}

We now modify the stack $\mathbb{X}$ by considering only $A$-modules which are stable under
$\nu_d$. To set things up properly, we need the following:

\begin{prop} \label{pXdd}
Fix dimension vectors $\vec{d}_0, \ldots, \vec{d}_n \in \mathbb{N}^{Q_0}$. Let
$\X^{\vec{d}_{\bullet}}$ be the full subcategory of $\X$ which, over a commutative noetherian ring
$R$, consists of flat families $\calm$ of $A$-modules over $R$ of dimension vector $\vec{d}$ such
that $h_i(\nu\calm)$ is a flat family of $A$-modules over $R$ with  dimension vector
$\vec{d}_i$ for all $i=0,\dots,n$.
Then $\X^{\vec{d}_{\bullet}}$ is a locally closed substack of $\X$.
\end{prop}
\begin{proof} 
We prove this by induction on $n$. Let $e_v$ be the idempotent corresponding to the vertex $v\in Q_0$.
Suppose $n=0$ and let $\mathcal{E}_v:=
h_0(\nu\calm)e_v$. Each $\mathcal{E}_v$ is a finitely generated $R$-module
and hence the locus $Z_v\subseteq \spec R$, where it is locally free of rank $d_{0,v}$, is locally closed in $R$;
in fact, it is given by the intersection of the closed condition determined by the $(d_{0,v}-1)$-st
Fitting ideal of ${\mathcal{E}_v}$ and the open condition determined by its $d_{0,v}$-th Fitting
ideal (see \cite[\href{http://stacks.math.columbia.edu/tag/07Z6}{Tag 07Z6}]{Stacks}). 
For $\X^{\vec{d}_{\bullet}}$ to be locally closed we need to show that given any
 base change $R \to S$, we have that $S\otimes_R\calm \in
\X^{\vec{d}_{\bullet}}(S)$ if and only if the map $\spec S \to \spec R$ factors through the locally
closed subscheme $\cap_v Z_v$. However, Fitting ideals commute with arbitrary base change and 
the formation of $\mathcal{E}_v$ also commutes with base change by Lemma \ref{lserrebchange}. Thus 
$h_0(\nu\calm)e_v$ is locally free of rank $d_{0,v}$ on a locally closed subscheme of $\spec R$.

Suppose now the proposition is true for $i=0,\dots,n-1$ i.e. there exists a locally closed subscheme
of $\spec R$ on which  $h_i(\nu\calm)e_v$ are locally free of rank $d_{i,v}$ for
$i=0,\dots,n-1$. By the same argument as in the base case, there exists a locally closed subscheme
of this (locally closed) subscheme on which $h_n(\nu\calm)e_v$ are locally free for all
$v\in Q_0$. 
Since  $h_i(\nu\calm)e_v$ are locally free, and hence flat,
Lemma \ref{lserrebchange} once again insures that
the formation of these modules commutes with base change. Finally, as before, Fitting ideals commute
with the base change as does the formation of subschemes.
\end{proof}

We now fix the {\em shift parameter} $d$ which will in examples be the dimension of the moduli space considered. 
Since we are interested in families $\calm$ such that $\nu_d\calm \simeq \calm$, we will need $\vec{d}$ 
to be a fixed point of the {\em (shifted)
Coxeter transformation} $\Phi: K_0(A) \to K_0(A)$ which the shifted Serre functor $\nu_d$ induces. We call such a dimension vector {\em Coxeter stable}. It is
also natural to restrict to only those families $\calm$ of $A$-modules such that
$h_i(\nu\calm)= 0$ for $i \neq d$ and $h_d(\nu\calm)$ is locally free of rank
vector $\vec{d}$. Proposition \ref{pXdd} ensures that the subcategory $\X^0 \subseteq \X$ 
of such $A$-modules is a locally closed substack.

\begin{cor} Let $s={\rm pd_A}DA <\infty$.
If $d=s-1$ or $s$ and $\vec{d}$ is Coxeter stable, then $\X^0$ is an open substack of $\X$.
\end{cor}
\begin{proof}
We need to examine the proof of Proposition \ref{pXdd} a little more carefully. For $i<s-1$, the locally
closed condition given by the vanishing of $h_i(\nu\calm)$ is open. Let us suppose that $\calm$ lies in this 
open substack.
Then for every residue field
$\kappa$ of $R$, we know that $\nu_d ( \kappa\otimes_R \calm)$ has homology only in degrees $0$ and
$1$.
Hence the dimension vector of $h_0(\nu_d(\kappa\otimes_R \calm))$ must be at least $\vec{d}$. Hence the locus where it is locally free of rank vector
$\vec{d}$ is open. Furthermore, if $d=s-1$ then on this open locus
$h_s(\kappa\otimes_R \calm\otimes^\L_ADA)=0$ since the dimension vector is fixed.
\end{proof}

We now define the shifted Serre functor in families. Let $T$ be a noetherian test scheme and $\calm \in \X^0(T)$. The functorial assignment
$\calm \mapsto h_0(\nu_d\calm)$ is compatible with base change by Lemma \ref{lserrebchange}
so it defines a morphism of stacks $\nu_d\colon \X^0\to \X$. We thus have both a diagonal map 
$\Delta\colon \X \to \X \times \X$ and a graph of $\nu_d$ morphism $\Gamma_{\nu_d}\colon \X^0 \to \X \times \X$. 
The natural definition of the moduli stack of representations fixed by the Serre functor 
is of course the fibre product 
\[ \X^S_{\vec{d}, d} = \X \times_{\X \times \X} \X^0.\]
We call this the \emph{moduli stack of Serre stable modules} or \emph{Serre stable representations} of dimension vector $\vec{d}$ and shift parameter $d$. If $\vec{d}, d$ are understood, we will drop the subscripts $\vec{d}, d$ and call this the {\em Serre stable moduli stack}. Note that $\X$ is not usually 
quasi-separated so $\Delta$ may not even be quasi-compact let alone a closed immersion. 
This gives us hope that we may be dealing with a more useful stack than $\X$. Note that given a
morphism $\calm \to \caln$ in $\X^0(T)$, we obtain a morphism $\nu_d\calm \to \nu_d\caln$ in $\X(T)$.
Also, products of Artin stacks are Artin stacks. Unravelling the definition of fibre products of stacks gives

\begin{prop} \label{pXS}
The objects of $\X^S(T)$ are equivalences classes of isomorphisms $\calm \to \call \otimes_T
\nu_d\calm$ where $\calm$ is a flat family of $A$-modules over $T$, $\call$ is a line bundle on $T$, and equivalence classes are 
defined as for $\X$ above in Subsection~\ref{subsec:rigid}. The morphisms
in $\X^S(T)$ consist of isomorphisms $\phi\colon\calm \to \calm'$  which are compatible in $\X(T)$ with
the isomorphisms $\theta\colon\calm \to \call \otimes_T \nu_d\calm,
\theta'\colon \calm' \to  \call' \otimes_T \nu_d\calm'$,  that is, there is an isomorphism of line bundles $\lambda: \call \lm \call'$ such that the diagram
$$\begin{CD}
\calm @>{\theta}>> \call \otimes_T \nu_d \calm \\
  @V{\phi}VV  @VV{\lambda\otimes_T \nu_d \phi}V \\
\calm' @>{\theta'}>> \call' \otimes_T \nu_d \calm' 
\end{CD}$$
commutes up to a scalar in $\calo_T^{\times}$. Furthermore, $\X^S$ is an Artin stack of finite type over $k$. 
\end{prop}

\begin{rem}
	Suppose that we have a flat family of $A$-modules $\calm \in \X(T)$ and an isomorphism
$\theta\colon \calm \to \call \otimes_T \nu_d\calm$
representing an object of $\X^S(T)$. Given an automorphism $\psi\colon \calm \to \calm$, 
we obtain a new object $(\call \otimes_T
\nu_d\psi^{-1}) \theta \psi\colon \calm \to \call \otimes_T \nu_d\calm$ which is isomorphic 
to $\theta$ in $\X^S(T)$. It will frequently be useful to pass to such an isomorphic family.
\end{rem}

\section{The $k$-points of $\X^S$}  \label{sec:kpoints}

In this section, we study the category of $k$-points $\X^S(k)$ of the Serre stable moduli stack $\X^S$. This will not only elucidate the stacky structure, but will also be invaluable for invoking Bridgeland-King-Reid theory in our goal of finding module-theoretic characterisations of endomorphism algebras of tilting bundles on smooth projective stacks. We also introduce the notion of regular semisimplicity which plays an important role in our criterion for tilting proved in Section~\ref{sTistilting}. 

We continue the notation of Section~\ref{sec:serrestable}. In particular, we will have fixed a dimension vector $\vec{d} \in K_0(A)$ and shift parameter $d$ and let $\X^S$ denote the corresponding Serre stable moduli stack. 
By Proposition~\ref{pXS}, an object of $\X^S(k)$ consists of a Serre stable module $M$ of dimension vector $\vec{d}$ together with an isomorphism $\theta: M \lm \nu_d M$. A priori, the object depends on $\theta$, but the next result gives a sufficient criterion for this not to be the case. This allows us to think of the $k$-points of $\X^S$ as being parametrised by the Serre stable modules themselves in this case. 

Recall that $K_0(A)$ has a natural partial ordering $\leq$ where modules induce non-negative elements. We say $\vec{d}\in K_0(A)$ is {\em minimal Coxeter stable} if $\vec{d}$ is Coxeter stable but the only other Coxeter stable $\vec{c} \in K_0(A)$ with $\vec{0} \leq \vec{c} \leq \vec{d}$ is $\vec{c} =  \vec{0}$.

\begin{prop}  \label{pisDM} 
Let $\vec{d}$ be a minimal Coxeter stable dimension vector and $M$ be a Serre stable module with dimension vector $\vec{d}$. Suppose that $\End_A M$ is a semisimple $k$-algebra with say $n$ Wedderburn components. 
\begin{enumerate}
\item $\End_A M \simeq k^n$ and the indecomposable summands $M_1,\ldots, M_n$ of $M$ form a single $\nu_d$-orbit. 
\item Any two isomorphisms $\theta: M \lm \nu_d M, \theta': M \lm \nu_d M$ define isomorphic objects in $\X^S(k)$. 
\item Let $\theta: M \lm \nu_d M$ be an isomorphism so $(M,\theta)$ defines an object of $\X^S(k)$. Then the automorphism group of $(M,\theta)$ in $\X^S(k)$ is $\mu_n$. 
\end{enumerate}
In particular, if every Serre stable module with dimension vector $\vec{d}$ has semisimple endomorphism ring, then $\X^S$ is a Deligne-Mumford stack whose $k$-points are parametrised by Serre stable modules.
\end{prop}
\begin{proof}
	The Deligne-Mumford criterion \cite[Th\'eor\`eme~8.1]{LM-B} will show that $\X^S$ is a
	Deligne-Mumford stack once we have proven parts (i), (ii) and (iii). Consider an object $\theta\colon M \xrightarrow{\sim} \nu_d M$ of $\X^S(k)$. Let $M =\oplus_{i=1}^n M_i$ be the decomposition into indecomposable summands. The Serre
functor is additive so $\nu_d M_i$ is a module, which given the isomorphism $\theta$ must be
isomorphic to one of the $M_j$'s. Minimality of $\vec{d}$ ensures that the $M_i$ are non-isomorphic and form a
single $\nu_d$-orbit so we can re-order them so that $M_{i+1} \simeq \nu_d M_i$.  Hence $\End_A M = k^n$ and the group of $A$-module automorphisms of $M$ is $(k^{\times})^n$. Part (i) is proved. 

We now prove (ii). If $\lambda = (\lambda_1,\ldots,\lambda_n) \in k^n$
defines an endomorphism of $M$ then $\nu_d\lambda = (\lambda_n,\lambda_1,\ldots,\lambda_{n-1})$. 
An easy computation now shows that as $\lambda$ ranges over $(k^{\times})^n$, $(\nu_d\lambda) \lambda^{-1}$ ranges over all $(\beta_1,\ldots,\beta_n) \in (k^{\times})^n$ with $\beta_1\beta_2 \ldots \beta_n = 1$. Up to scalar, this covers all of $(k^{\times})^n$ so any other
isomorphism $\theta': M \xrightarrow{\sim} \nu_d M$ gives an object isomorphic to $\theta \colon M \to \nu_d M$.

Finally, to prove (iii), note that $\nu_d\colon k^n \to k^n$ has $n$
$1$-dimensional eigenspaces, the eigenvalues being the elements of $\mu_n$. Each eigenspace gives a 
unique automorphism of the object $\theta\colon M
\simeq \nu_d M$ in $\X^S(k)$ and there are no other automorphisms. Composition of automorphisms corresponds to multiplication of eigenvalues so we are done. 
\end{proof}

For canonical algebras, we will be interested in the dimension vector $\vec{d}= \mathbf{1}$, where all components are 1. The hypotheses of the previous proposition are easily fulfilled in this case. 
\begin{prop}  \label{ponevector} 
Any module $M$ with dimension vector $\mathbf{1}$ has endomorphism ring $\End_A M \simeq k^n$ where $n$ is the number of indecomposable summands of $M$. 
\end{prop}
\begin{proof}
Consider the decomposition into indecomposables $M = \oplus_i M_i$. The choice of dimension vector means that $\Hom_A(M_i,M_j) = 0$ if $j \neq i$. Furthermore, given any endomorphism $f\colon M_i \to M_i$ let $N = \im f$. Looking at each vertex, one sees easily that the composite $N \hookrightarrow
M_i \xrightarrow{f} N$ is an isomorphism since $\vec{d}$ only has entries 1. Thus $f$ surjects onto some direct summand and so must either be 0 or an isomorphism.
\end{proof}

One of our goals is to see if we can recognise endomorphism algebras of a tilting bundles on a smooth projective stacks. Such an algebra $A$ has finite global dimension, so we shall assume this for the rest of this section. One case that has been studied in some detail are Lenzing-Meltzer's \cite{LM} {\em concealed-canonical algebras} which are precisely the  endomorphism algebras of tilting bundles on weighted projective lines $\Y$.

For our study of concealed canonical algebras, we need the bilinear {\em Euler form} on $K_0(A)$, namely
$$ \langle [M], [N] \rangle = \sum_i (-1)^i \dim \Ext^i_A(M,N) .$$
Suppose now that $d=1$ and $\vec{d}$ is Coxeter stable. Then by Serre duality for $\D^b(\Amod)$, we have $\langle \vec{d}, - \rangle = - \langle -,\vec{d} \rangle$. In particular, $\langle \vec{d}, \vec{d} \rangle = 0$ which suggests that $\langle \vec{d}, - \rangle: K_0(A) \lm \Z$ is a natural stability condition to use. This motivates
\begin{defn}
With the above hypotheses, we say that an $A$-module $M$ is {\em regular simple} if $\langle \vec{d}, [M] \rangle = 0$ but $\langle \vec{d}, [N] \rangle < 0$ for every proper submodule $N$ of $M$. We say that an $A$-module is {\em regular semisimple} if it is a direct sum of regular simples. 
\end{defn}

If $A$ is basic, tame hereditary, that is, the quiver algebra of an extended Dynkin diagram, then the Coxeter stable vectors are all multiples of the imaginary root $\vec{\delta}$ and setting $\vec{d} = \vec{\delta}$, the above definition coincides with the usual definition of regular simple. 

We are interested in regular semisimplicity because of the following standard orthogonality result from stability theory which is easily proven. It will be used to apply Bridgeland-King-Reid theory. 

\begin{prop}  \label{pendreg} 
Let $N,N'$ be two non-isomorphic regular simple modules. Then $\Hom_A(N,N') = 0$ and $\End_A N = k$. In particular, if $M$ is regular semisimple, then $\End_A M$ is semisimple too. 
\end{prop}

Regular simple modules arise naturally in the following context.

\begin{prop}  \label{prop:kpregsimple}
 Let $\calt= \oplus_{v \in Q_0} \calt_v$ be a tilting bundle on a weighted projective line $\Y$ where the $\calt_v$ are non-isomorphic indecomposable summands. Let $A = \End_{\Y} \calt$ be the associated concealed canonical algebra whose quiver has vertices $Q_0$ and $\vec{d}$ be the dimension vector of $\calt$. For any simple sheaf $S$ on $\Y$, the $A$-module $M = \Rhom_{\Y}(\calt,S)$ is regular simple. 
\end{prop}
\begin{proof}
 Let $p$ be a generic point of the scheme locus of $\Y$ so $N = \Rhom_{\Y}(\calt, k(p))$ is an $A$-module with dimension vector $\vec{d}$. We may assume that $S, k(p)$ have disjoint support so the derived equivalence $\Rhom_{\Y}(\calt,-)$ shows that $\langle \vec{d}, [M] \rangle = 0$. 
 Let $M'$ be a proper submodule $M$ with $\langle \vec{d} , [M'] \rangle \geq 0$. We may assume that $M'$ is indecomposable. Let $S' = M' \otimes^{\L}_{A} \calt$. Now $\coh \Y$ is hereditary so $S' = F[i]$ for some indecomposable sheaf $F$ and $i\geq 0$. Also, $(-) \otimes^{\L}_A \calt$ is a derived equivalence so the inclusion $M' \hookrightarrow M$ shows that $0 \neq \Hom_{\D^b}(F[i], S) = \Ext^{-i}_{\Y}(F,S)$. Thus $i=0$. If $F$ is a bundle, then $\Ext^1_{\Y}(F, k(p)) = 0$ so $\langle \vec{d} , [M'] \rangle < 0$. If $F$ is a finite sheaf, then applying $\Rhom_{\Y}(\calt,-)$ to the exact sequence 
$$ 0 \lm \ker \iota \lm F \xrightarrow{\iota} S \lm 0 $$
we see that $M'$ is not a submodule of $M$. 
\end{proof}

\section{Canonical algebras} \label{scanonical} 

In this section, we will look at Ringel's \cite{R} canonical algebras $A$ afresh, starting from the generators and relations description. In other words, we will pretend that the algebra was given to us ``at random'' and not assume we know for example that it is the endomorphism ring of a tilting bundle on a weighted projective line. We follow instead the non-commutative algebro-geometric theme announced in the introduction, and consider studying $A$ through moduli spaces. In particular, we will compute the Serre stable moduli stack $\X^S$ for a natural choice of dimension vector $\vec{d}$ and show it is a weighted projective line. 

Let $A$ be the {\em canonical algebra} given by the following quiver $Q=(Q_0, Q_1)$
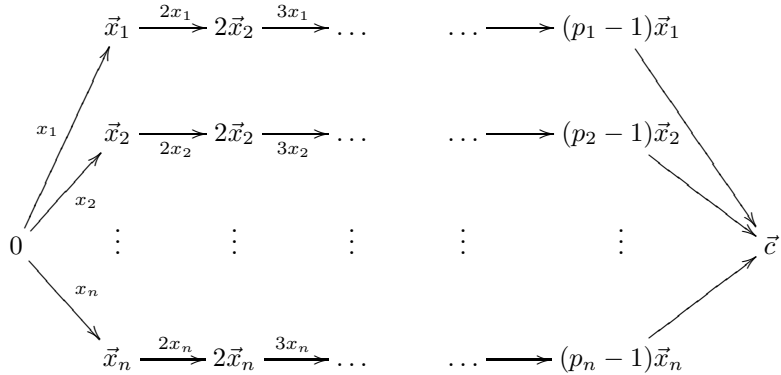
\[\xymatrix{
		&\x_1\ar[r]^{2x_1}&2\x_2\ar[r]^{3x_1}&\dots&\dots\ar[r]&
		(p_1-1)\x_{1}\ar[ddr]\\
		&\x_2\ar[r]_{2x_2}&2\x_2\ar[r]_{3x_2}&\dots&\dots\ar[r]
		&(p_2-1)\x_2\ar[dr]\\
		0\ar[uur]^{x_1}\ar[ur]_{x_2}\ar[dr]^{x_n}&\vdots&\vdots&\vdots&\vdots&\vdots&\c\\
		&\x_n\ar[r]^{2x_n}&2\x_n\ar[r]^{3x_n}&\dots&\dots\ar[r]&(p_n-1)\x_n
		\ar[ur]
	}\] 
	where  $n\geq 2, p_i\geq 1$, with relations
	\[x_i^{p_i}=x_1^{p_1}+\lambda_i x_2^{p_2}, \quad \text{for } i\ge 3, \lambda_i \neq 0.\] 
If $\Y$ is the weighted projective line with weight $p_1$ at 0, $p_2$ at $\infty$ and $p_i$ at
$\lambda_i$ for $i\ge 3$, then Geigle and Lenzing show in \cite{GL} that there is a derived equivalence
$\D^b(\coh\Y)\simeq \D^b(\Amod)$. We will independently reprove this result in Section~\ref{sTistilting}. 

The {\em tubular} canonical algebras are the ones whose weights satisfy $\sum_i (1 - \frac{1}{p_i}) = 2$. These are the ones for which $\nu_1^e \simeq \id$ for some positive integer $e$ (or equivalently, $\omega_{\Y}^{\otimes e} \simeq \calo_{\Y}$).  

\begin{prop}  \label{poneismin}
$\vec{d} = \mathbf{1}$ is a  minimal Coxeter stable dimension vector (for the shift parameter $d=1$). Furthermore, it is unique if $A$ is not tubular. 
\end{prop}
\begin{proof}
This follows from any one of several possible elementary computations that we omit. We note however, that it can easily be verified from first principles without resort to invoking any special theory (such as Proposition~\ref{p:regssforcc}).  Indeed, the indecomposable projectives form a basis for $K_0(A)$ and if $P_{\vec{x}}$ denotes the projective corresponding to the vertex $\vec{x}$, then $\nu(P_{\vec{x}}) = I_{\vec{x}}$ where $I_{\vec{x}}$ is the indecomposable injective corresponding to $\vec{x}$. Hence the Coxeter transformation $\Phi$ is readily computed. In the non-tubular case, one needs only check that $\mathbf{1}$ spans the eigenspace $E_1$ of $\Phi$ with eigenvalue 1. In the tubular case, one needs first to show that this eigenspace contains $\mathbf{1}$ and is 2-dimensional.  If it were not minimal, then there would be another Coxeter stable dimension vector $\vec{d'}$ all of whose entries are 0 or 1. Writing $I = \{ v \in Q_0| d'_v = 1\}$ we see then that every vector $\vec{b} \in E_1$ has the same co-ordinates on $I$ and the same on $Q_0 - I$. We thus need only find another Coxeter stable vector with 3 distinct co-ordinates. For example, the tubular canonical algebra with weights $(p_1,p_2,p_3) = (4,4,2)$ has a Coxeter stable dimension vector 
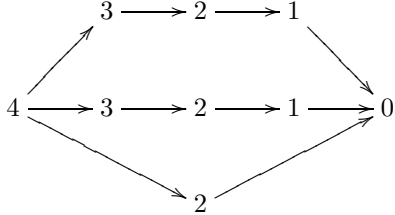
\[\xymatrix{
		& 3 \ar[r] & 2  \ar[r] & 1 \ar[dr] & \\ 
4 \ar[ur] \ar[r] \ar[drr] & 3 \ar[r] & 2  \ar[r] & 1 \ar[r] & 0\\
		& & 2  \ar[urr] &  &  
	}\]
The other tubular canonical algebras have a similar Coxeter stable dimension vector.
\end{proof}
This proposition suggests that it might be interesting to study the moduli stack $\X^S$ of Serre stable $A$-modules with dimension vector $\vec{d} = {\bf 1}$ and shift parameter $d=1$. By Propositions~\ref{pisDM}, \ref{ponevector}, a $k$-point of this stack is given by a representation
$M=(M_v,\varphi_{\alpha})_{v\in Q_0,\alpha\in Q_1}$ of $Q$, satisfying the above relations, with dimension vector ${\bf 1}$, and which is Serre stable in the sense that 
\[M\simeq M\otimes_A^\L DA[-1]\] 
Equivalently, Serre stability here means $\Tor^A_1(M,DA)=D\Ext^1_A(M, A)\simeq
M$ and $\Tor^A_i(M,DA)=D\Ext^i_A(M, A) = 0$ if $i\neq1$.

It will be instructive to examine closely the $k$-points of $\X^S$ in order to have a good guess as to what $\X^S$ is. The following lemmas will be helpful in this regard. 
\begin{lemma} \label{lem:cond}
	Let $M$ be a $k$-point of $\X^S$. Then for any other
	representation $N$ we have for all $i$: \[\Ext^i_A(M,N)\simeq D\Ext^{1-i}_A(N,M).\]
	In particular $\Hom_A(M,N)\simeq D\Ext^1_A(N,M)$ and $\Ext^1_A(M,N)=D\Hom_A(N,M)$.
\end{lemma}
\begin{proof}
	We have $\Hom_\D(M,N[i])=D\Hom_\D(N[i],\nu M)=D\Hom_\D(N,\nu M[-i])$ and since $M\simeq \nu M[-1]$ we
	have \[\Hom_\D(M,N[i])\simeq D\Hom_\D(N, M[1-i])\] from which the result follows. 
\end{proof}
\begin{lemma}  \label{lformofkpt} 
	Let $M=(M_v,\varphi_{\alpha})_{v\in Q_0,\alpha\in Q_1}$ be a $k$-point of $\X^S$. If $\varphi_{ax_i}=0$ for some $a=1,\dots,p_i$ then
	$\varphi_{bx_i}=0$ for all
	$b=1,\dots,p_i$. Furthermore, if $\varphi_{x_i}=0$ then $\varphi_{x_j}\neq 0$ if $i\neq j$.
\end{lemma}
\begin{proof}
	Suppose $a\neq p_i$ and that $\varphi_{ax_i}=0$. We claim that $\varphi_{(a+1)x_i}=0$.
	Note that $\Hom_A(M, S_{a\x_i})\simeq k$ where $S_{a\x_i}$ is the simple representation at vertex
	$a\x_i$. We now compute $\Ext^1_A(S_{a\x_i},M)$. We have the following projective
	resolution \[0\lm P_{(a+1)\x_i}\lm P_{a\x_i}\lm\ S_{a\x_i}\lm 0\] where $P_{a\x_i}$ is
	the projective at vertex $a\x_i$, which gives the following
	long exact sequence\[0\lm\Hom_A(S_{a\x_i},M)\lm \Hom_A(P_{a\x_i},M)\xrightarrow{f}
	\Hom_A(P_{(a+1)\x_i},M)\lm \Ext^1_A(S_{a\x_i},M)\lm 0.\] Now
	$\Hom_A(P_{a\x_i},M)\simeq \Hom_A(P_{(a+1)\x_i},M)\simeq k$ and so
	if  $\varphi_{(a+1)x_i}\neq0$, then $\Hom_A(S_{a\x_i},M)=0$ implying $f$ is surjective
	and forcing $\Ext^1_A(S_{a\x_i},M)=0$. This contradicts, Lemma \ref{lem:cond} and so
	$\varphi_{(a+1)x_i}=0$.

	A similar argument, using injective resolutions shows that if $a\neq 1$ and $\varphi_{ax_i}=0$ then
	$\varphi_{(a-1)x_i}=0$.

	Note if $\varphi_{x_i}=\varphi_{x_j}=0$ then the relations, and the result we just proved, imply all arrows must
	be $0$. Such a representation cannot be Serre stable, for it contains a projective summand,
	namely $P_{\c}$.
\end{proof}
This lemma shows that $\X^S$ is covered by open substacks $\mathbb{U}_{i}$ consisting of Serre stable families where only $\varphi_{x_i}$ is allowed to be $0$. Furthermore, any $k$-point of say $\mathbb{U}_n$ is given by a module isomorphic to one of the form

\begin{equation}
 \xymatrix{
		&k\ar[r]^{1}&k\ar[r]^{1}&\dots&\dots\ar[r]^{1}&
		k\ar[ddr]^{1}\\
		&k\ar[r]_{1}&k\ar[r]_{1}&\dots&\dots\ar[r]_{1}
		&k\ar[dr]_{x}\\
		k\ar[uur]^{1}\ar[ur]_{1}\ar[dr]^{y}&\vdots&\vdots&\vdots&\vdots&\vdots&k\\
		&k\ar[r]^{y}&k\ar[r]^{y}&\dots&\dots\ar[r]^{y}&k
		\ar[ur]^{y}
	}
\label{ekpointofUn}
\end{equation} 
where $x \neq 0, \lambda_3,\ldots, \lambda_{n-1}$ and $y^{p_n} = x - \lambda_n$. This is indecomposable for $y \neq 0$ and the direct sum of $p_n$ indecomposables when $y = 0$. From Proposition~\ref{pisDM}, we see that the corresponding inertia groups are 1 and $\mu_{p_n}$. This immediately suggests the following

\begin{thm}\label{thm:main}
There is an isomorphism $\X^S\simeq \Y$.
\end{thm}
\begin{proof}
We first compute the open substack $\mathbb{U}_{n}$. The other open patches $\mathbb{U}_i$ can be computed similarly. It will then be clear from the computations below, that they glue together to form $\Y$ as desired. 
\begin{lemma}\label{calc}
	Let $T$ be a scheme. Elements of $\mathbb{U}_n(T)$, are isomorphic to precisely those of the form
	$\calm\simeq L_{p_n-1}^{-1}\otimes_T\nu_1\calm$ where 
\[\calm:=\quad\quad
\begin{minipage}{10pt}
		\xymatrix{
		&\calo\ar[r]^{1}&\calo\ar[r]^{1}&\dots&\dots\ar[r]^{1}&
		\calo\ar[ddr]^{1}\\
		&\calo\ar[r]_{1}&\calo\ar[r]_{1}&\dots&\dots\ar[r]_{1}
		&\calo\ar[dr]_{x}\\
		\calo\ar[uur]^{1}\ar[ur]_{1}\ar[dr]^{a_1}&\vdots&\vdots&\vdots&\vdots&\vdots&\calo\\
		&L_1\ar[r]^{a_2}&L_2\ar[r]^{a_3}&\dots&\dots\ar[r]^{a_{p_n-1}}&L_{p_n-1}
		\ar[ur]^{a_{p_n}}
	}
	\end{minipage},
	\] $\calo=\calo_T$ and $L_i$ are invertible sheaves on $T$.
\end{lemma}

\begin{proof} An element of $\mathbb{U}_n(T)$ consists of a representation 
\[\calm:=\quad\quad
\begin{minipage}{10pt}
		\xymatrix{
			&L_{\x_1}\ar[r]^{\varphi_{2x_1}}
			&L_{2\x_1}\ar[r]^{\varphi_{3x_1}}
			&\dots&\dots\ar[r]
			& L_{(p_1-1)\x_1}\ar[ddr]^{\varphi_{p_1x_1}}\\
			&L_{\x_2}\ar[r]_{\varphi_{2x_2}}
			&L_{2\x_2}\ar[r]_{\varphi_{3x_2}}
			& \dots&\dots\ar[r]
			&L_{(p_2-1)\x_2}\ar[dr]_{\varphi_{p_2x_2}}\\
			L_0\ar[uur]^{\varphi_{x_1}}\ar[ur]_{\varphi_{x_2}}\ar[dr]^{\varphi_{x_n}}
			&\vdots&\vdots&\vdots&\vdots&\vdots&L_{\c}\\
			&L_{\x_n}\ar[r]^{\varphi_{2x_n}}
			&L_{2\x_n}\ar[r]^{\varphi_{3x_n}}
			&\dots&\dots\ar[r]
			&L_{(p_n-1)\x_n} \ar[ur]^{\varphi_{p_nx_n}}
	}
	\end{minipage}
\] together with an isomorphism $\calm\simeq N\otimes_T\nu_1\calm$ 
 where $N, L_{a\vec{x}_i}$ are invertible sheaves on $T$.
By applying $L_{0}^{-1}\otimes_T -$ to obtain an isomorphic element of $\mathbb{U}_n(T)$, we can assume
$L_0=\calo$. Furthermore, for $i=1,\dots, n-1$ all $\varphi_{x_i}$ must be surjective for otherwise, there
will be an element of $\mathbb{U}_n(k)$ with $\varphi_{x_i}=0$ which contradicts the definitions of $\mathbb{U}_n$.
Thus, these $\varphi_{x_i}$ must be isomorphisms and so $L_{a\x_i}\simeq \calo$ and by
further changing bases if necessary, we see that $\calm$ is of the form given.

Now we compute $\nu_1\calm$.
Let  $P_{a\x_i}$ be the projective at vertex $a\x_i$ and recall
	that $P_{a\x_i}\otimes_A DA=I_{a\x_i}$ which is the injective at vertex $a\x_i$. 

	We have the following resolution 
	\[0\lm 
		\begin{bmatrix}
			\calo\otimes P_{\x_n}\\
			L_1\otimes P_{2\x_n}\\
			\vdots\\
			L_{(p_n-1)}\otimes P_{\c}
		\end{bmatrix}
		\xrightarrow{\partial}
		\begin{bmatrix}
			\calo\otimes P_0\\
			L_1\otimes P_{\x_n}\\
			\vdots\\
			L_{p_n-1}\otimes P_{(p_n-1)\x_n}
		\end{bmatrix}
	\lm \calm \lm 0, \quad \partial=
	\begin{bmatrix}
		1\otimes x_n&0&\dots &-a_{p_n}\otimes x_2^{p_2}\\
		-a_1\otimes 1&1\otimes2x_n\\
		0&-a_2\otimes 1\\
		\vdots&\vdots\\
		0&0&\dots& 1\otimes p_nx_n
	\end{bmatrix}
\]
We now apply $-\otimes_ADA$ to the resolution to obtain the complex
\[0\lm
\begin{bmatrix}
			\calo\otimes I_{\x_n}\\
			L_1\otimes I_{2\x_n}\\
			\vdots\\
			L_{(p_n-1)}\otimes I_{\c}
		\end{bmatrix}
		\xrightarrow{\partial}
		\begin{bmatrix}
			\calo\otimes I_0\\
			L_1\otimes I_{\x_n}\\
			\vdots\\
			L_{p_n-1}\otimes I_{(p_n-1)\x_n}
		\end{bmatrix}\lm 0
\] It is easy to see that $\partial$ is surjective and hence there is homology only in degree $1$. We
compute the kernel and obtain
\[
\nu_1\calm=\quad\quad
\begin{minipage}{10pt}
		\xymatrix{
		&L_{p_n-1}\ar[r]^{1}&L_{p_n-1}\ar[r]^{1}&\dots&\dots\ar[r]^{1}&
		L_{p_n-1}\ar[ddr]^{1}\\
		&L_{p_n-1}\ar[r]_{1}&L_{p_n-1}\ar[r]_{1}&\dots&\dots\ar[r]_{1}
		&L_{p_n-1}\ar[dr]_{1}\\
		L_{p_n-1}\ar[uur]^{1}\ar[ur]_{x}\ar[dr]^{a_{p_n}}&\vdots&\vdots&\vdots&\vdots&\vdots&L_{p_n-1}\\
		&\calo\ar[r]^{a_1}&L_1\ar[r]^{a_2}&\dots&\dots\ar[r]^{a_{p_n-2}}&L_{p_n-2}
		\ar[ur]^{a_{p_{n-1}}}
	}
	\end{minipage}
\]Thus, by fixing an isomorphism $\phi\colon \calm\simeq L_{p_n-1}^{-1}\otimes_T\nu_1\calm$ we
obtain an element of $\X^S(T)$.
\end{proof}
The modules in (\ref{ekpointofUn}) suggest what the universal representation on $\mathbb{U}_n$ looks like. First, let \[R:=\frac{k[x^{\pm 1}, (x-\lambda_3)^{-1},\dots,(x-\lambda_{n-1})^{-1}, y]}{y^{p_n}=x-\lambda_n}\]
Note that 
$\mu_{p_n}$ acts on $R$ by multiplying $y$ by a primitive $p_n$-th root of unity $\zeta$. In fact
$[\spec R/\mu_{p_n}]$ is
an open substack of $\Y$. We claim that $[\spec R/\mu_{p_n}]\simeq \mathbb{U}_{n}$. To show this consider the
following flat family $\tilde\calu$ of $A$-modules over $R$: 
\begin{equation}
 \xymatrix{
		&R\ar[r]^{1}&R\ar[r]^{1}&\dots&\dots\ar[r]^{1}&
		R\ar[ddr]^{1}\\
		&R\ar[r]_{1}&R\ar[r]_{1}&\dots&\dots\ar[r]_{1}
		&R\ar[dr]_{x}\\
		R\ar[uur]^{1}\ar[ur]_{1}\ar[dr]^{y}&\vdots&\vdots&\vdots&\vdots&\vdots&R\\
		&R\ar[r]^{y}&R\ar[r]^{y}&\dots&\dots\ar[r]^{y}&R
		\ar[ur]^{y}
	}
\label{eUR}
\end{equation} 
	This family is $\mu_{p_n}$-equivariant (as replacing $y$ with $\zeta^iy$ yields an 
	isomorphic family) and hence is a family over $[\spec R/\mu_{p_n}]$. From the previous
	lemma, we see that indeed $\nu_1\tilde\calu\simeq\tilde\calu$ and so $\tilde\calu\in \X^S(R)$ 
	and thus we get a map
	$[\spec R/\mu_{p_n}]\to \mathbb{U}_n$. We claim that the family is universal and hence the map is in fact an
isomorphism. 
More precisely, we aim to show that any $\calm\in \mathbb{U}_n(T)$, which must have the form as described in Lemma
\ref{calc}, is a pullback of $\tilde\calu$ via a unique morphism $T\to[\spec R/\mu_{p_n}]$.
	From the proof of Lemma \ref{calc}, we see that $\phi$ induces isomorphisms \[L_1\otimes_T L_{p_n-1}\simeq \calo, \quad L_2\otimes_T L_{p_n-1}\simeq L_1,\quad\dots\quad,L_{p_n-1}\otimes_T
L_{p_n-1}\simeq L_{p_n-2}\] thus \[L_1^{\otimes 2}\simeq L_2, \quad L_1^{\otimes 3}\simeq
L_3,\dots\quad L_1^{\otimes p_n}\simeq \calo.\] 
Hence $L_1$ is a $p_n$-torsion line bundle, which together with the isomorphism $L_1^{\otimes p_n}\simeq \calo$ defines an \'etale cyclic cover $\pi\colon\tilde{T} = \underline{\spec}_T \left( \bigoplus_{i=0}^{p_n-1} L^{\otimes i} \right)\to T$.
We thus get 
\[ 
\pi^*\calm=\quad\quad
\begin{minipage}{10pt}
		\xymatrix{
		&\calot\ar[r]^{1}&\calot\ar[r]^{1}&\dots&\dots\ar[r]^{1}&
		\calot\ar[ddr]^{1}\\
		&\calot\ar[r]_{1}&\calot\ar[r]_{1}&\dots&\dots\ar[r]_{1}
		&\calot\ar[dr]_{x}\\
		\calot\ar[uur]^{1}\ar[ur]_{1}\ar[dr]^{a_1}&\vdots&\vdots&\vdots&\vdots&\vdots&\calot\\
		&\calot\ar[r]^{a_1}&\calot\ar[r]^{a_1}&\dots&\dots\ar[r]^{a_1}&\calot
		\ar[ur]^{a_1}
	}
	\end{minipage}
\] where $\calo=\calo_{\tilde T}$. This family comes as a pullback of $\tilde\calu$ via the map
		$R\to \calo_{\tilde T}$ given by $x\mapsto x$ and
	$y\mapsto a_1$.
This verifies the isomorphism $\mathbb{U}_n \simeq [\spec R/ \mu_{p_n}]$. Patching together the $\mathbb{U}_i$ using (\ref{eUR}) now shows that $\X^S \simeq \Y$ as desired. 
\end{proof}

\vspace{3mm}

Let $\mathcal{U}$ be the universal representation on $\X^S$ and $\mathcal{T}:= \mathcal{U}^{\vee}$ be the dual bundle. We will recover Geigle-Lenzing's derived equivalence in Section~\ref{sTistilting}, by showing that $\mathcal{T}$ is a tilting bundle on $\X^S$.  
In preparation for this we need one more result about the $k$-points of $\X^S$. 
\begin{prop}  \label{pcanregsemi}
Let $M$ be a $k$-point of $\X^S$. Then $M$ is regular semisimple. 
\end{prop}
\begin{proof}
First note that the dimension vector $\mathbf{1}$ can be expressed as $\mathbf{1} = [P_{\vec{0}}] - [P_{\vec{c}}]$.  From Lemma~\ref{lformofkpt} or (\ref{ekpointofUn}), we know that the indecomposable summands $M_i$ of $M$ are either simple or have a unique top $S_0$ and unique socle $S_{\vec{c}}$. Thus if $N$ is any proper submodule of $M_i$, then there exists a subset $I \subseteq \{\x_1,2\x_2, \ldots, (p_n-1)\x_n\}$ such that 
$$[N] = [S_{\vec{c}}] + \sum_{\vec{x} \in I} [S_{\x}]. $$
We compute then that 
$$ \langle \mathbf{1}, [N] \rangle = \langle - [P_{\vec{c}}],[S_{\vec{c}}] \rangle = -1 < 0.$$
Hence $M$ is regular semisimple. 
\end{proof}

It would be an interesting exercise to compute the Serre stable moduli stack for other minimal Coxeter stable dimension vectors in the tubular case. The Coxeter stable dimension vectors which are $\geq 0$ are easy enough to describe. For the tubular algebra with weights $(4,4,2)$, these are all $\mathbb{Q}_{\geq 0}$-linear combinations of the two vectors
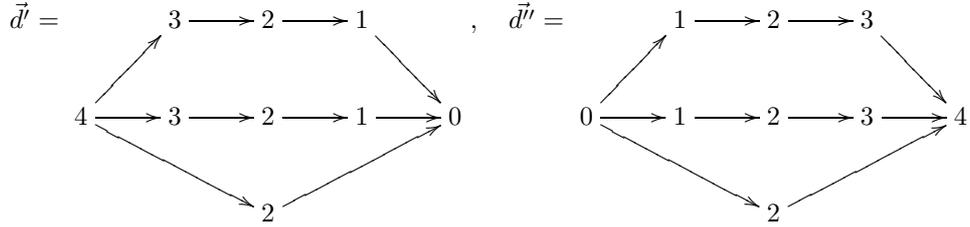
\[\vec{d'} = \xymatrix{
		& 3 \ar[r] & 2  \ar[r] & 1 \ar[dr] & \\ 
4 \ar[ur] \ar[r] \ar[drr] & 3 \ar[r] & 2  \ar[r] & 1 \ar[r] & 0\\
		& & 2  \ar[urr] &  &  
	}, \quad 
\vec{d''} = \xymatrix{
		& 1 \ar[r] & 2  \ar[r] & 3 \ar[dr] & \\ 
0 \ar[ur] \ar[r] \ar[drr] & 1 \ar[r] & 2  \ar[r] & 3 \ar[r] & 4\\
		& & 2  \ar[urr] &  &  
	}\]
The other cases are obtained similarly. The dimension vector $\vec{d'}$ above corresponds to a vector bundle on $\Y$ under the Geigle-Lenzing derived equivalence. We suspect that the Serre stable moduli stack for this dimension vector is still $\Y$. This would parallel the fact that the moduli space of skyscraper sheaves on an elliptic curve $E$ is isomorphic to $E$, which is in turn isomorphic to any moduli space of line bundles on $E$ with fixed degree. 
	\section{Beilinson algebra}\label{sec:Beil}
In this section, we examine the Beilinson algebra $B_n$, which is given by the following quiver 
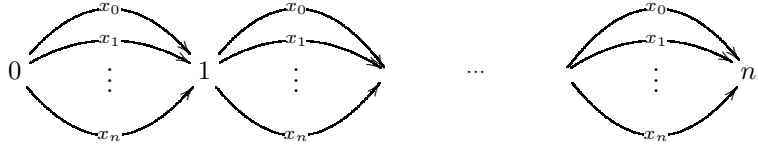
\[\xymatrix@C=60pt{0\ar@/^2pc/[r]|{x_0}\ar@/^1pc/[r]|{x_1}
\ar@{}[r]|{\vdots}
\ar@/_2pc/[r]|{x_{n}}
&1
\ar@/^2pc/[r]|{x_0}\ar@/^1pc/[r]|{x_1}
\ar@{}[r]|{\vdots}
\ar@/_2pc/[r]|{x_{n}}
&\ar@{}[r]|{\dots}&
\ar@/^2pc/[r]|{x_0}\ar@/^1pc/[r]|{x_1}
\ar@{}[r]|{\vdots}
\ar@/_2pc/[r]|{x_{n}}
&n
}
\] 
with commutative relations $x_i x_j = x_j x_i$ for all $i,j$. We will compare our approach using the Serre stable moduli stack with the traditional quiver GIT approach for the case $\vec{d} = \mathbf{1}, d=n$. 

In the quiver GIT approach, one has to pick a stability condition, and here the usual one used is the linear map $\rho: K_0(B_n) \lm \Z$ defined by $\rho([M]) = \dim M e_n - \dim M e_0$ where $e_i \in B_n$ denotes the primitive idempotent corresponding to the vertex $i$. One considers only {\em $\rho$-stable} modules with dimension vector $\bf{1}$, which by definition means only those $M$ such that for any proper submodule $N$ of $M$, we have $\rho([N]) >0$. One readily sees that these are precisely the representations of the form
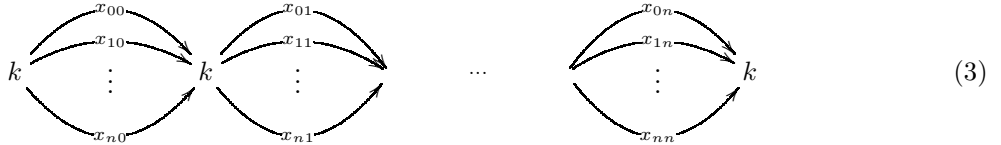
\begin{equation}
\xymatrix@C=60pt{k\ar@/^2pc/[r]|{x_{00}}\ar@/^1pc/[r]|{x_{10}}
\ar@{}[r]|{\vdots}
\ar@/_2pc/[r]|{x_{n0}}
&k
\ar@/^2pc/[r]|{x_{01}}\ar@/^1pc/[r]|{x_{11}}
\ar@{}[r]|{\vdots}
\ar@/_2pc/[r]|{x_{n1}}
&\ar@{}[r]|{\dots}&
\ar@/^2pc/[r]|{x_{0n}}\ar@/^1pc/[r]|{x_{1n}}
\ar@{}[r]|{\vdots}
\ar@/_2pc/[r]|{x_{nn}}
&k
}
\label{eBeik}
\end{equation}
where $\vec{x}_i := (x_{0i}, \ldots, x_{ni}) \neq \vec{0}$ for all $i$. This is an open condition. In this case, the commutativitiy relations show that all $\vec{x}_i$ are proportional and define a common point of $\PP^n$. If $\X^{\rho}_{\bf{1}}$ denotes the open substack of $\X_{\bf 1}$ consisting of $\rho$-stable modules, then quiver GIT in this case gives the moduli space $\X^{\rho}_{\bf 1}$ which is naturally isomorphic to $\PP^n$ (this will also be clear from our calculation of the Serre stable moduli stack below). 

We now turn to the Serre stable moduli stack. 
\begin{prop}  \label{prop:serreisrho}
Let $M$ be a Serre stable $B_n$-module with dimension vector $\bf 1$. Then $M$ is $\rho$-stable.
\end{prop}
\begin{proof}
First note that if $S_i$ denotes the simple $B_n$-module at vertex $i$, then a simple downward induction shows $\pd S_i = n-i$. 
Suppose now that $M$ is not $\rho$-stable so in the notation of (\ref{eBeik}) $\vec{x}_i = \vec{0}$ say. Then there exists a direct sum decomposition $M = M_{\leq i} \oplus M_{>i}$ of $B_n$-modules such that $\pd M_{>i} < n$. It follows that $\nu_n(M_{> i})$ is not a module so $M$ cannot be Serre stable. 
\end{proof}

\begin{thm}\label{thm:Beil}
	We have $\X^S_{{\bf 1}, n}=\PP^{n}$. Furthermore, if $\calu$ is the universal representation, then $\nu_n \calu \simeq \omega_{\PP^n} \otimes_{\PP^n} \calu$. 
\end{thm}
\begin{proof}
	Consider the following family $\calu$ of $B_n$-modules over $\PP^{n}_{x_0:\dots:x_n}$:
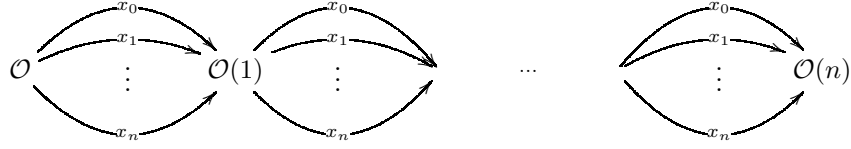
	\[\xymatrix@C=60pt{\calo\ar@/^2pc/[r]|{x_0}\ar@/^1pc/[r]|{x_1}
	\ar@{}[r]|{\vdots}
	\ar@/_2pc/[r]|{x_{n}}
	&\calo(1)
	\ar@/^2pc/[r]|{x_0}\ar@/^1pc/[r]|{x_1}
	\ar@{}[r]|{\vdots}
	\ar@/_2pc/[r]|{x_{n}}
	&\ar@{}[r]|{\dots}&
	\ar@/^2pc/[r]|{x_0}\ar@/^1pc/[r]|{x_1}
	\ar@{}[r]|{\vdots}
	\ar@/_2pc/[r]|{x_{n}}
	&\calo(n)
	}
	\]
Recall that $\calu$ is the universal representation on $\X^{\rho}_{\bf{1}} \simeq \PP^n$ and that its dual $\calt = \calu^{\vee}$ is a tilting bundle on $\PP^n$ with endomorphism ring isomorphic to $B_n$. 

We first show that $\calu$ is Serre stable so picking any isomorphism $\theta:\calu \simeq \omega_{\PP^n}^{-1} \otimes_{\PP^n} \nu_n \calu$ determines a morphism $\PP^n \lm \X^S$. Let $P_i$ denote the indecomposable projective $B_n$-module at vertex $i$ and $I_i$ denote the injective at $i$. We need the following
\begin{prop}\label{prop:ref} 
We have the following exact sequences of $B_n$-modules over $\PP^n$. 
\begin{gather}
0\lm \Omega^{n}(n)\otimes P_n \lm \Omega^{n-1}(n-1)\otimes P_{n-1} \lm \dots \lm
\Omega^1(1)\otimes P_1 \lm \calo_{\PP^{n}}\otimes P_0 \lm \calu\lm 0 \label{eBeiP}\\
0 \lm \omega_{\PP^n} \otimes_{\PP^n} \calu \lm \Omega^{n}(n)\otimes I_n \lm \Omega^{n-1}(n-1)\otimes I_{n-1} \lm \dots \lm
	\Omega^1(1)\otimes I_1 \lm \calo_{\PP^{n}} \otimes I_0  \lm 0 \label{eBeiI}
\end{gather}
Furthermore, the second sequence is, up to twisting by a line bundle, the dual of the left module version of the first sequence. 
\end{prop}
\begin{proof}
We have from \cite{Be}, the following resolution of the diagonal $\Delta\colon
\PP^{n}\to\PP^{n}\times\PP^{n}$:
\[0\lm p^*\Omega^{n}(n)\otimes q^*\O(-n)\lm\dots\lm p^*\Omega^{1}(1) \otimes q^*\O(-1)
\lm \O_{\PP^{n}\times \PP^{n}}\lm \O_\Delta\lm 0\] where
\[\xymatrix{\PP^n\times \PP^n\ar[r]^p\ar[d]^q&\PP^n\ar[d]^f\\
\PP^n\ar[r]^u&{\rm pt}}\]
$p$ is the projection onto the first factor
and $q$ is the projection onto the second. We now apply the functor $p_*\shom(q^*\calt
,-)$
to this exact sequence. Since
\begin{align*}
R^jp_*\shom(q^*\calt, p^*\Omega^i(i)\otimes q^*\O(-i))&=R^jp_*(p^*\Omega^i(i)\otimes q^*\shom(\calt,\calo(-i))) \\
		&=\Omega^i(i)\otimes_{\PP^n}R^jp_*q^*\shom(\calt,\calo(-i))\quad\text{ \cite[Ch. 3,
		Ex. 8.2]{Har}}\\
		&=\Omega^i(i)\otimes_{\PP^n} f^*R^ju_*\shom(\calt,\calo(-i))\quad\text{\cite[Proposition
		9.3]{Har}}\\
		&=\Omega^i(i) \otimes \Ext^j_{\PP^n}(\calt,\calo(-i))\\
		&=
		\begin{cases}
			\Omega^i(i)\otimes P_i&{\rm if}\;j=0\\
			0& \rm otherwise
		\end{cases}
	\end{align*}
we see that the functor preserves the exactness of the Beilinson resolution and
furthermore, transforms it to the exact sequence (\ref{eBeiP}) since 
$p_*\shom(q^*\calt,\calo_{\Delta})\simeq \shom(\calt,\calo)=\calu$.

Since (\ref{eBeiP}) is an exact complex of locally free sheaves, applying $\shom(-, \calo(-1))$ to it gives the following exact sequence of left $B_n$-modules over $\PP^n$
\[ 0 \lm \omega_{\PP^n} \otimes_{\PP^n} \calu \lm I_0'\otimes_k\Omega^{n}(n)  \lm  I_1'\otimes_k\Omega^{n-1}(n-1) \lm\dots\lm I_{n-1}'\otimes_k\Omega^1(1) \lm   I_n'\otimes_k\calo \lm 0 \]
where $I_i'$ is the injective left ${B_n}$-module at vertex $i$. Now $B_n \simeq B^{op}_n$ though the isomorphism switches the idempotent at vertex $i$ with that at vertex $n-i$. Using this isomorphism gives (\ref{eBeiI}). 
\end{proof} 
Serre stability of $\calu$ is now easily verified. Applying $-\otimes_{B_n}DB_n$ to the resolution of $\calu$ in (\ref{eBeiP}) gives the coresolution of $\omega_{\PP^n} \otimes_{\PP^n} \calu$ in (\ref{eBeiI}). Hence $\nu_n(\calu)\simeq \omega_{\PP^n} \otimes_{\PP^n} \calu$ as desired. 

We wish to show that any isomorphism $\calu \xrightarrow{\sim}  \omega_{\PP^n}^{-1} \otimes_{\PP^n} \nu_n \calu$ defines the universal family on $\X^S$. To this end, consider a test scheme $T$ and let $(\calm, \theta) \in \X^S(T)$
	 be given by the family of $B_n$-modules 
	 \[
\calm := 
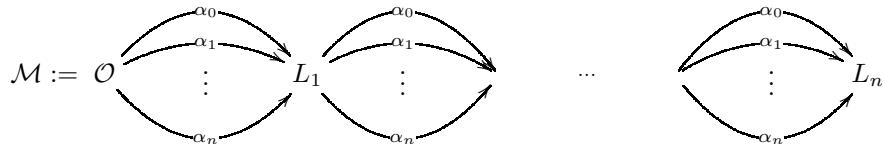
\xymatrix@C=60pt{\calo\ar@/^2pc/[r]|{\alpha_0}\ar@/^1pc/[r]|{\alpha_1}
\ar@{}[r]|{\vdots}
\ar@/_2pc/[r]|{\alpha_{n}}
&L_1
\ar@/^2pc/[r]|{\alpha_0}\ar@/^1pc/[r]|{\alpha_1}
\ar@{}[r]|{\vdots}
\ar@/_2pc/[r]|{\alpha_{n}}
&\ar@{}[r]|{\dots}&
\ar@/^2pc/[r]|{\alpha_0}\ar@/^1pc/[r]|{\alpha_1}
\ar@{}[r]|{\vdots}
\ar@/_2pc/[r]|{\alpha_{n}}
&L_n
}
\]  
and isomorphism $\theta:\calm \xrightarrow{\sim} L \otimes_T \nu_n \calm$ where $L, L_i$ are invertible sheaves on $T$. Now $\calu$ is the universal family on $\X^{\rho}_{\bf{1}}$ and $\calm$ is $\rho$-stable by Proposition~\ref{prop:serreisrho}, so there is a unique map $f\colon Y\to \PP^n $ such that $\calm = f^*\calu$. Hence it suffices to show that the $T$-point $(\calm,\theta)$ is independent of $\theta$. Suppose we vary $\theta$ by pre-composing with an automorphism $\psi = (\psi_0,\ldots,\psi_n)$ where $\psi_i \in \Aut L_i = H^0(\calo_T^{\times})$. Now $\rho$-stability implies that the sections $\alpha_1,\dots,\alpha_n$ generate $L_1$ so we see in fact that we must have all the $\psi_i$ are equal and an automorphism $\psi$ of $\calm$ is just given by an element of $H^0(\calo_T^{\times})$. Since we are working in the rigidified moduli stack, varying by $\psi$ does not alter the $T$-point. 
\end{proof} 

This example of the Beilinson algebra exhibits several phenomena arising in the theory of Serre stable moduli stacks that are worth emphasising. The first is that unlike in the theory of quiver GIT, we do not need to pick a separate stability condition, Serre stability already gives $\rho$-stability. Furthermore, we have

\begin{prop}  \label{prop:BeiminCox}
$\vec{d} = \bf 1$ is the unique minimal Coxeter stable dimension vector.
\end{prop}
\begin{proof}
As in Proposition~\ref{poneismin}, one can readily check this from first principles. The easiest way to see the result is to use Beilinson's derived equivalence or more precisely, the ensuing isomorphism $K_0(B_n) \simeq K_0(\PP^n)$. If $H^i \subset \PP^n$ denotes the intersection of $i$ generic hyperplanes in $\PP^n$, then 
$$ [\calo], [\calo_{H}], [\calo_{H^2}], \ldots , [\calo_{H^n}] $$
is a basis for $K_0(\PP^n)$ and the (shifted) Coxeter transformation $\Phi$ sends $\calo_{H^i}$ to 
$$[\omega_{\PP^n} \otimes_{\PP^n} \calo_{H^i}] = [\calo_{H^i}] - (n+1) [\calo_{H^{i+1}}].$$
Hence $\Phi$, with respect to the basis above is  single Jordan block with eigenvalue 1.
\end{proof}
\noindent
The proposition shows that the choice of dimension vector $\bf 1$ is also the only natural one for Serre stable moduli stacks. Hence, there is much less floppiness in the theory here. 

The universal sheaf $\calu$ on $\X^S$ is the same as for $\X^{\rho}_{\bf 1}$ so its dual is a tilting bundle inducing Beilinson's derived equivalence $\D^b(\coh \PP^1) \simeq \D^b(\mo B_n)$ \cite{Be}. In particular, we see that the Serre stable moduli stack is useful even for algebras which are not quasi-tilted, although in this case, the shift parameter may be different from 1.

\section{Serre functor for cyclic quotient stacks and orders}  \label{sec:serrestacks}

The Serre stable moduli stack can be defined in other settings too as long as there is a suitable notion of a Serre functor in families. For stacks and orders, this can be done using the notion of the canonical (bi)-module. In this section, we collect some basic facts about canonical sheaves and bimodules that we need for the rest of this paper. The Serre stability condition is most illuminating in the case of cyclic quotient stacks, so we examine it in this context. 

Fix a cyclic group $G = \langle \s \rangle$ of order $n$. Let $Y$ be a smooth (quasi-projective) variety. Consider a {\em $G$-cover} of $Y$, which for us will simply mean a cover of the form $\tilde{Y} = \underline{\spec}_Y \tilde{\calo}$ where
$$ \tilde{\calo} = \bigoplus_{\chi \in G^{\vee}} \call_{\chi} $$
and $G$ acts on the line bundle $\call_{\chi}$ by the character $\chi$ of $G$. We will say that $\tilde{Y}$ is {\em 1-generated} if $\tilde{\calo}$ is generated as an $\calo$-algebra by a single eigensheaf, say $\call_{\chi}$. In this case, we will also say $\tilde{Y}$ is {\em $\chi$-generated}. The cyclic covering trick allows us to construct such a $G$-cover. Given any effective divisor $D$ on $Y$ and line bundle $\call_{\chi}$ with an isomorphism $m:\call_{\chi}^n \simeq \calo_Y(-D)$, the algebra $\oplus_{i = 0}^{n-1} \call_{\chi}^i$ defines a $\chi$-generated $G$-cover if we use $m$ to define the multiplication (see \cite[Definition~2.50]{KM}). This $G$-cover is \'etale on the complement of $D$. If $D$ is smooth, then so is $\tilde{Y}$. 

If $\Y$ is a $d$-dimensional smooth Deligne-Mumford stack of finite type, then it has a canonical sheaf $\omega_{\Y}$ and $\nu_d :=\omega_{\Y} \otimes_{\Y} (-)$ is an auto-equivalence of $\coh \Y$ which serves as our shifted Serre functor. 
To understand this functor, we specialise to the case where $\tilde{Y}$ is a smooth $G$-cover of a smooth variety $Y$. Then coherent sheaves on the stack $\Y = [\tilde{Y}/G]$ can be viewed as $G$-equivariant sheaves on $Y$ or equivalently, left modules over the skew group ring $\cala:=\calo_{\tilde{Y}} \# G$, which is a finite sheaf of algebras on $Y$. This gives an equivalence of categories $\coh \Y \simeq \cala\mod$. The canonical sheaf then corresponds to the $G$-equivariant sheaf $\omega_{\tilde{Y}}$. 

We wish to understand $\nu_d$ in this special context by considering the induced auto-equivalence of $\cala\mod$ which we shall also denote by $\nu_d$. 
Let $\calo = \calo_Y, \tilde{\calo} = \calo_{\tilde{Y}}, \w:=\w_{Y}$ and $\tilde \w:=\w_{\tilde Y}$. Recall that the adjunction formula gives
$$ \tilde{\w} = \shom_Y(\tilde{\calo}, \w) \simeq \shom_Y(\tilde{\calo}, \calo) \otimes_Y \w .$$

Consider the trace map ${\rm tr} : k(\tilde{Y}) \lm k(Y)$  where $k(Y),k(\tilde{Y})$ are the function fields  of $Y$ and $\tilde{Y}$. The trace pairing $k(\tilde{Y}) \times k(\tilde{Y}) \lm k(Y):(a,b) \mapsto \rm{tr} (ab)$ is non-degenerate, so we may use it to identify the dual sheaf $\shom_Y(\tilde{\calo}, \calo)$ with a subsheaf of $k(\tilde{Y})$ and hence $\tilde{\w}$ with a subsheaf of $k(\tilde{Y}) \otimes_Y \w$. Now $\cala$ is an order in the matrix $k(Y)$-algebra $\cala \otimes_Y k(Y)$. Following the custom in non-commutative algebraic geometry, we mimic the adjunction formula and define the {\em canonical $\cala$-bimodule} to be 
\[\w_{\cala} = {\mathcal Hom}_Y(\cala,\w) \simeq {\mathcal Hom}_Y(\cala,\calo) \otimes_Y \w.\]
Then $\nu_d = \w_{\cala} \otimes_{\cala} (-): \cala\mod \lm \cala\mod$. This is a well-known fact which follows easily from Lemma~\ref{lem:omega}(ii) below. 
As for $\tilde{\w}$, we may use the (reduced) trace map ${\rm tr}\colon\cala \otimes_Y k(Y) \to k(Y)$ to identify $\Hom_Y(\cala, \calo_Y)$ with a sub-bimodule of $\cala \otimes_Y k(Y)$. Now $\cala \otimes_Y k(Y)$ naturally contains $k(\tilde{Y})$ so $\w_{\cala}$ naturally contains $\tilde{\w}$ too.

\begin{lemma} \label{lem:omega}
Let $\tilde{Y}$ be a $\chi$-generated $G$-cover of $Y$.
\begin{enumerate}
\item Let $\tilde{\w}_{\chi}$ be the eigensheaf (on $Y$) of $\tilde{\w}$ corresponding to the character $\chi$. Then multiplication induces an isomorphism $\tilde{\w} \simeq \tilde{\w}_{\chi} \otimes_Y \calot$.  
\item The bimodule multiplication map induces isomorphisms $\w_{\cala} \simeq \tilde{\w} \otimes_{\tilde{Y}} \cala \simeq \cala \otimes_{\tilde{Y}} \tilde{\w}$. 
\end{enumerate}
\end{lemma}
\begin{proof}
We may work locally over $Y$ and so assume $\calot = \calo[y]/(y^n -f) =  \O\oplus\O y \dots\oplus \O y^{n-1}$ where $G$ acts on $y$ via the character $\chi$. 
It follows that
	\begin{align*}
	\shom_{\O}(\tilde\O,\O)&=\O\oplus \O y^{-1}\oplus\dots\oplus \O y^{1-n}\\
	&=y^{1-n}\tilde\O.
\end{align*}
Hence $\tilde\w=y^{1-n}\tilde\O\otimes_\O\w$ and (i) follows since $\tilde{\w}_{\chi} = y^{1-n}\w$. 

It remains only to show that $\w_{\cala} =y^{1-n}\cala\otimes_{\O}\w$. To verify this, we need a matrix embedding of $\cala$ which can easily be obtained from a Peirce decomposition as follows. Recall $kG=\!\!  \prod\limits_{\mu \in G^{\vee}}\! k\eps_{\mu}$ and
$\eps_{\mu}$ is the primitive idempotent corresponding to the character $\mu$. Then $\cala = {\bigoplus\limits_{\mu,\lambda} \eps_{\mu} \cala \eps_{\lambda}}$ and ordering the elements of $G^{\vee}$ appropriately, we obtain the following algebra homomorphism $\iota: \cala\hookrightarrow M_n(\O)$ which is compatible with the Peirce decomposition. For $g \in \calo$, 
\[
g\mapsto 
\begin{pmatrix}
	g&0&\dots &0\\
	0&g&\dots &0\\
	\vdots&&\ddots\\
	0&0&\dots&g
\end{pmatrix}\quad
\sigma\mapsto
\begin{pmatrix}
	\zeta&0&\dots &0\\
	0&\zeta^2&\dots &0\\
	\vdots&&\ddots\\
	0&0&\dots&\zeta^n
\end{pmatrix}\quad
y\mapsto
\begin{pmatrix}
	0&0&\dots&\dots&f\\
	1&0&\dots&\dots&0\\
	0&1&\dots&\dots&0\\
	0&0&\dots&\ddots&0\\
	0&0&\dots&1&0
\end{pmatrix}
\]
The image of $\iota$ is easily seen to be 
\begin{equation}\label{eqn:Amatrix}
\begin{pmatrix} 
	\calo &  (f) & \dots& (f) \\
	\calo & \calo & &   \vdots \\
	\vdots&\vdots&\ddots& (f) \\
	\calo & \dots & \dots&\calo 
\end{pmatrix}
\end{equation}
The trace pairing is now easily computed, allowing us to identify
\begin{align*}
	\shom_{\O}(\cala,\O)&=\{g\in M_n(k(\O))\mid {\rm tr}(g\cala)=0\}\\
	&=
\begin{pmatrix} 
	\calo &  \calo & \dots&\calo \\
	(f^{-1}) & \calo & &\calo \\
	\vdots&\vdots&\ddots&\vdots\\
	(f^{-1}) & (f^{-1}) & \dots&\calo 
\end{pmatrix}\subset M_n(k(\O))\\
&=\cala
\begin{pmatrix} 
	0 &  \dots & 0&1 \\
	f^{-1} & 0&\dots & 0 \\
	0&f^{-1}&\dots&0\\
	\vdots & &\ddots &0\\
	0&\dots&f^{-1}&0
\end{pmatrix}\\
&=\cala yf^{-1}=\cala y^{1-n}
\end{align*} from which (ii) follows.
\end{proof}

Recall that if $Y$ is a smooth variety, then any skyscraper sheaf $k(p)$ is Serre stable in the sense that $\w_Y \otimes_Y k(p)\simeq k(p)$. This in part motivated Bondal and Orlov's definition of point objects. For stacks, simple sheaves are not necessarily Serre stable. 

\begin{eg}[Serre unstable simple sheaves on a smooth Deligne-Mumford stack.]
\label{eg}

Let $Y= \spec \calo$ be an affine curve and suppose $p \in C$ is a closed point defined by the local
equation $f=0$ where $f \in \calo$. We consider the cover $\tilde{Y}  = \spec \calot$ where 
$\calot = \calo[y]/(y^n -f)$ which is totally ramified over $p$ and unramified elsewhere. 
Hence $\Y = [\tilde{Y}/G]$ is a smooth Deligne-Mumford stack with a single stacky point above $p$ with inertia group $G$ and elsewhere is isomorphic to $Y-p$. 

Let $\mu \in G^{\vee}$ and $\eps_{\mu}\in kG$ be the corresponding primitive idempotent. Then $P_{\mu} = \cala \eps_{\mu}$ is an indecomposable $\cala$-module corresponding to a ``column'' of the matrix form of $\cala$ in (\ref{eqn:Amatrix}) above. The  $P_{\mu}$ are pairwise
non-isomorphic, being the projective covers of non-isomorphic simple modules $S_{\mu} =
P_{\mu}/yP_{\mu}$. Now $P_{\mu}$ is an $(\cala,\calo)$-bimodule which is free as a right
$\calo$-module. We may thus view $P_{\mu}$ as a flat family of $\cala$-modules over $Y$. Furthermore, away from $p$, the corresponding family of sheaves on $\Y$ is just the tautological sheaf $\calo_{\Y}$ on $\Y$. Hence, the $P_{\mu}$ give $n$ flat families of $\cala$-modules over $Y$ which are all isomorphic away from $p$, but different at $p$. Informally speaking, this shows that  the rigidified moduli stack $\X$ (of sheaves on $\Y$ with the same discrete invariant as a generic skyscraper sheaf) is non-separated above $p$ and does not look like $\Y$. 

We now see how these families are Serre unstable. From our local computations in the proof of Lemma~\ref{lem:omega}, we see  
$$\nu_d (P_{\mu}) =\w_{\cala}\otimes_{\cala} P_{\mu} \simeq yP_{\mu}.$$ 
Thus $\nu_d$ permutes the $P_{\mu}$ and $S_{\mu}$ cyclically. The instability of the $P_{\mu}$ is caused by the Serre unstable $\cala$-module $P_{\mu} \otimes_{\calo} k(p)$ which is a non-split extension of all the $S_{\mu}$. As can be expected and will be seen later, in the Serre stable moduli stack, stable
reduction will replace this unstable fibre with the Serre stable module
$\bigoplus\limits_{\mu\in G^\vee} S_{\mu}$. This stable reduction will require passing to the ramified $G$-cover, $\tilde{Y}$ of $Y$. This completes the example.
\end{eg}

\begin{notn}  \label{nkp}
Let $\Y$ be a weighted projective curve with coarse moduli scheme $C$ and $p \in C$ be a closed point. Locally in a neighbourhood of $p$, we may write $\Y = [\tilde{Y}/G]$ in the notation of Example~\ref{eg}. We let $k_{\Y}(p)$ denote the Serre stable sheaf $\bigoplus\limits_{\mu\in G^\vee} S_{\mu}$ on $\Y$. It is the direct sum of $n$ simple sheaves where $n$ is the weight of $\Y$ at $p$. 
\end{notn}

\section{The dual of the universal bundle is tilting}  \label{sTistilting}

In this section, we apply the tilting theory of \cite{BKR}, to give a module+moduli-theoretic criterion for the dual of the universal sheaf on the Serre stable moduli stack to be tilting. This allows us to recover Geigle-Lenzing's derived equivalence \cite{GL} for canonical algebras. 

Let $A$ be a finite dimensional algebra which is basic and {\em connected} in the sense that its quiver is connected. We fix the shift parameter $d=1$ and a minimal Coxeter stable dimension vector $\vec{d}$ and form the corresponding Serre stable moduli stack $\X^S$ and universal sheaf $\mathcal{U}$. We assume that $\text{gl.dim}\ A < \infty$ so that $A$ has a chance of being concealed-canonical. 

We will unfortunately need to assume that we know $\X^S$ is a weighted projective curve. To check this abstractly, it suffices to show it is a smooth proper irreducible Deligne-Mumford stack of finite type over $k$. We saw this moduli-theoretic condition holds for canonical algebras $A$ Theorem~\ref{thm:main}. For concealed canonical algebras, this also should follow from general moduli principles, though we do not have the required stack technology at present to prove this. For example, given any $k$-point $M\simeq \nu_1 M$ of $\X^S$, then (see proof of Theorem~\ref{tBKR} below) $\pd M = 1$ so  $\Ext^2_A(M,M)=0$. Hence at least the corresponding point of $\X$ is smooth.  We know from Proposition~\ref{p:regssforcc} that the minimal Coxeter stable dimension vector is unique, and generically $M$ is regular simple so $\Ext^1_A(M,M) = k$. Hence at least $\X$ is 1-dimensional. Presumably, a closer analysis will show that the same results hold true for $\X^S$. What seems hard to prove is that some form of stable reduction holds and hence that $\X^S$ is proper. One of the goals of ongoing research is to replace our moduli-theoretic assumption here with a module-theoretic one, something which usually occurs in the tilting theory of \cite{BKR}. 

Let $C$ be the coarse moduli scheme of $\X^S$ and using Notation~\ref{nkp} we let 
$$\mathcal{S} = \{ k_{\X^S}(p) | \text{closed } p \in C\} .$$
Let $\mathcal{T} = \mathcal{U}^{\vee}$ which we can view as an $(A,\calo_{\X^S})$-bimodule. We wish of course to show that the functor 
$$F:= \Rhom_{\X^S}(\mathcal{T},-) = 
\mathbf{R}\Gamma( (-) \otimes_{\X^S} \mathcal{U}): \D^b(\X^S) \lm \D^b(A) $$
is a derived equivalence. If $\calm$ is a coherent sheaf on $\X^S$ that is supported on a finite set, then we will abuse notation and write $F \calm = \calm \otimes_{\X^S} \mathcal{U}$. 

It is instructive to describe explicitly the modules $Fk_{\X^S}(p)$. We may compute  $Fk_{\X^S}(p)$ locally in a neighbourhood of $p$ and so assume that $\X^S$ is the cyclic group quotient $[\tilde{Y}/G]$ in the notation of Example~\ref{eg}. In this language, $\mathcal{U}$ is a $G$-equivariant sheaf on $\tilde{Y}$. Let $\tilde{p}$ be the point of $\tilde{Y}$ lying above $p$ which we assume as in Example~\ref{eg} is fixed by $G$. Then $k_{\X^S}(p) = \oplus_{\mu \in G^{\vee}} k(\tilde{p})_{\mu}$ where $k(\tilde{p})_{\mu}$ is the skyscraper sheaf at $\tilde{p}$ with $G$-action given by the character $\mu$, and 
$$F k_{\X^S}(p) = \left(\bigoplus_{\mu \in G^{\vee}} k(\tilde{p})_{\mu} \otimes_{\tilde{Y}} \mathcal{U}\right)^G= k(\tilde{p}) \otimes_{\tilde{Y}} \mathcal{U} $$
where the superscript $G$ denotes $G$-invariants. In particular we see that $F k_{\X^S}(p)$ is a Serre stable $A$-module with dimension vector $\vec{d}$. 

We wish to apply Bridegeland-King-Reid's general criterion for an exact functor to be an equivalence. In our case, the version we need is the following.

\begin{lemma}  \label{lBKR} 
Suppose that $A$ is a basic connected finite dimensional algebra and $\vec{d}$ a dimension vector such that $\X^S$ is a weighted projective curve. Then $F$ is an equivalence of categories provided it induces isomorphisms
$$ F: \Ext^i_{\X^S}(k_{\X^S}(p),k_{\X^S}(p')) \lm 
\Ext^i_A(Fk_{\X^S}(p),Fk_{\X^S}(p')) $$  
for all $i$ and $p, p' \in C$. 
\end{lemma}
\begin{proof}
This is just a special case of \cite[Theorem~2.4]{BKR} and we need only check that the hypotheses there hold. First note that $\mathcal{S} = \{ k_{\X^S}(p) \}$ is a spanning class for $\D^b(\X^S)$ as is easily seen by repeating the proof of \cite[Example~2.2]{Br}. Connectedness of $A$ ensures indecomposability of $\D^b(A)$. Finally, recall that the $k_{\X^S}(p)$ are Serre stable by (\ref{nkp}) and so are the $Fk_{\X^S}(p)$ as we have just observed. Hence $F( \nu k_{\X^S}(p)) \simeq \nu ( Fk_{\X^S}(p))$ where $\nu$ is the Serre functor on $\D^b(\X^S)$ and $\D^b(A)$. 
\end{proof}

\begin{thm}  \label{tBKR} 
Let $A$ be a basic connected finite dimensional algebra of finite global dimension and $\vec{d}\in K_0(A)$ be a minimal Coxeter stable dimension vector. Suppose that
\begin{enumerate}
\item the Serre stable moduli stack $\X^S$ is a weighted projective curve, and
\item any Serre stable module $M$ of dimension vector $\vec{d}$ is regular semisimple.
\end{enumerate}
Then $\Rhom_{\X^S}(\mathcal{T},-): \D^b(\X^S) \lm \D^b(A)$ is an equivalence of categories, where $\mathcal{T}$ is the dual of the universal representation on $\X^S$. 
\end{thm}
\begin{proof}
We need only check the hypothesis of Lemma~\ref{lBKR} concerning isomorphisms of Ext groups. First note that $\coh \X^S$ is hereditary and together with Notation~\ref{nkp} we find 
$$
\Hom_{\X^S}(k_{\X^S}(p), k_{\X^S}(p')) = D\Ext^1_{\X^S}(k_{\X^S}(p'), k_{\X^S}(p)) =
\begin{cases}
k^{q_p} & \text{ if } p=p' \\
0 & \text{else} 
\end{cases}
$$
where $q_p$ is the weight of $\X^S$ at $p$. Let $M = Fk_{\X^S}(p), M' = Fk_{\X^S}(p')$. By Serre duality and stability we know that $\Ext^i_A(M,A) = 0$ if $i \neq 1$. As $\text{gl.dim}\ A < \infty$ we see that the projective dimension of $M$ is 1. In particular $\Ext^i_A(M,M') = 0$ if $i\geq 2$. If $p \neq p'$ then $M \not\simeq M'$ by Proposition~\ref{pisDM} and hence $\Hom_A(M,M') = 0$ by Proposition~\ref{pendreg} and our hypothesis~(ii). Serre duality and stability then shows $\Ext^1_A(M,M') = 0$. Finally, we consider the case $p=p'$. We know from Proposition~\ref{pisDM} that $M$ is a direct sum of $q_p$ regular simple modules and that these form a single $\nu_1$-orbit so say $M = \oplus_{i=0}^{q_p-1} \nu_1^i N$ for some regular simple $N$. In particular, $\End_A M$ is a $q_p$-dimensional vector space as is $\Ext^1_A(M,M)$ by Serre duality. Hence $F$ certainly induces an isomorphism 
$$ F: \Hom_{\X^S}(k_{\X^S}(p),k_{\X^S}(p)) \lm 
\Hom_A(Fk_{\X^S}(p),Fk_{\X^S}(p))$$
We may work locally near $p$ and so assume $\X^S = [\tilde{Y}/G]$ in the notation of Example~\ref{eg} where $G$ is cyclic of order $q_p$ and $k_{\X^S}(p) = \oplus_{i=0}^{q_p-1} \nu_1^i k(\tilde{p})$ for $\tilde{p} \in \tilde{Y}$ the point lying over $p$. 
It remains only to show that $F$ induces a non-zero map
$$ F: \Ext^1_{\X^S}(\nu_1^{i+1}k(\tilde{p}), \nu_1^{i} k(\tilde{p})) \lm 
\Ext^1_A(F\nu_1^{i+1}k(\tilde{p}),F\nu_1^{i}k(\tilde{p})) $$
We compute this ``Kodaira-Spencer'' map explicitly by constructing $\calo_{2\tilde{p}} \otimes_{\tilde{Y}}\mathcal{U}$ as follows. Serre duality gives a unique non-split extension 
$$ 0 \lm \nu_1 N \lm E \lm N \lm 0 .$$
Applying   powers of $\nu_1$ to this exact sequence and taking direct sums gives an extension
\begin{equation} 0 \lm M \lm \bar{\mathcal{U}} \lm M \lm 0
\label{ebarU}\end{equation}
where $\bar{\mathcal{U}} = \oplus_{i=0}^{q_p -1} \nu_1^i E$. This module is a Serre stable self-extension of $M$ so gives a flat family of Serre stable modules over the ring of dual numbers $k[\eps]$. This $k[\eps]$-point of $\X^S$ is given by a morphism $\phi: \spec k[\eps] \lm \tilde{Y}$. Now $\bar{\mathcal{U}}$ is non-split so $\phi$ is not constant and so must give an isomorphism of $k[\eps]$ with $\calo_{2\tilde{p}}$. Hence we may assume that $\calo_{2\tilde{p}} \otimes_{\tilde{Y}}\mathcal{U} \simeq \bar{\mathcal{U}}$. 
Consider now a non-split extension
$$ 0 \lm \nu_1^{i+1} k(\tilde{p}) \lm \calm_i \lm \nu_1^i k(\tilde{p}) \lm 0$$
which represents a non-zero element of $ \Ext^1_{\X^S}(\nu_1^{i+1}k(\tilde{p}), \nu_1^{i} k(\tilde{p})) $. Its image under $F$ is the extension 
$$ 0 \lm \nu_1^{i+1} k(\tilde{p}) \otimes_{\X^S}\bar{\mathcal{U}} \lm \calm_i \otimes_{\X^S}\bar{\mathcal{U}}\lm \nu_1^i k(\tilde{p}) \otimes_{\X^S}\bar{\mathcal{U}} \lm 0$$
which is non-split since it must be one of the $q_p$ direct summands of (\ref{ebarU}).  
\end{proof}

We immediately arrive at an independent proof of (a slight refinement of) Geigle-Lenzing's derived equivalence \cite{GL}. 
\begin{cor}  \label{c:GL}
Let $A$ be a canonical algebra, $\X^S_{\mathbf{1},1}$ be the Serre stable moduli stack and $\calu$ be the universal sheaf. Then $\calt = \calu^{\vee}$ is tilting and induces a derived equivalence between the weighted projective line $\X^S_{\mathbf{1},1}$ and $A$.
\end{cor}

Theorem~\ref{tBKR} raises some interesting questions. Firstly, can it be used to characterise concealed canonical algebras? Secondly, can it be used to give an independent proof of the following result of Lenzing and de la Pe\~{n}a?

\begin{thm}[Lenzing-de la Pe\~{n}a] \cite[Theorem~1.1]{LdP} 
\label{thm:LdP}
 A finite dimensional $k$-algebra $A$ is concealed-canonical if and only if $\mo A$ has a sincere separating exact subcategory $\mo_0 A$ (see \cite[Section~1]{LdP} for the definition). 
\end{thm}
We will not need the definition of a sincere separating exact subcategory and remark only that $\mo_0 A$ corresponds to the category of finite length sheaves on the associated weighted projective line. Now the category of finite length sheaves is stable under the shifted Serre functor $\omega_{\Y} \otimes_{\Y} (-)$ so $\mo_0 A$ is stable under $\nu_1$.

We now give some partial answers to the questions above. With this aim in mind, we will not assume the reverse implication of Theorem~\ref{thm:LdP} (proved in \cite[Section~6,7]{LdP}) although we will occasionally refer to results in \cite{LdP} which appear earlier in that paper. 

Recall that the rank of a coherent sheaf on a weighted projective line $\Y$ induces a {\em rank function} $\rk: K_0(A) \lm \Z$ where $A$ is any concealed canonical algebra derived equivalent to $\Y$. If $A$ is ``only' assumed to have a sincere separating exact subcategory, then this rank function can be defined independently of the derived equivalence as in \cite[Section~2, (S9)]{LdP}. 

\begin{prop}  \label{p:regssforcc}
 Let $A$ be a concealed canonical algebra. Then 
\begin{enumerate}
 \item $A$ has finite global dimension. 
 \item $K_0(A)$ has a unique minimal Coxeter stable dimension vector $\vec{d}$ of rank 0, 
 \item every Serre stable module with dimension vector $\vec{d}$ is regular semisimple and the generic one is regular simple. 
\end{enumerate}
\end{prop}
\noindent
\textbf{Remark} a) One can prove this directly quite easily (use Proposition~\ref{prop:kpregsimple} for part (iii)). We will prove this however assuming ``only'' that $A$ has a sincere separating exact subcategory $\mo_0 A$. 

\noindent 
b) If $A$ is {\em tubular} then there may be Coxeter stable dimension vectors of non-zero rank. However, this does not occur if $A$ is not tubular (see for example the proof of Proposition~\ref{poneismin}). 
\begin{proof}
(i) This is \cite[Section~2, (S6)(iii)]{LdP}.

 (ii) Let $\Phi: K_0(A) \lm K_0(A)$ be the Coxeter transformation. We know from \cite[Corollary~4.3]{LdP} that 1 is an eigenvalue of $\Phi$ of algebraic multiplicity 2. Now $\Phi$ commutes with the rank function \cite[Section~2, (S9)]{LdP}, so its restriction to $\ker \left(\rk: K_0(A) \lm \Z\right)$ still has 1 as an eigenvalue, but now with multiplicity 1. 

(iii) Let $M$ be a Serre stable module with dimension vector $\vec{d}$. Part (i) and \cite[Section~3, (S12)]{LdP} show that $-\langle \vec{d}, - \rangle$ is the rank function. Furthermore, we know from  \cite[Section~2, (S9)]{LdP}, that every submodule $N$ of $M$ satisfies $\langle \vec{d}, [N] \rangle \leq 0$ with equality if and only if it lies in $\mo_0 A$. From \cite[Section~2, (S8)]{LdP}, we know that $\mo_0 A$ is a coproduct of connected uniserial subcategories and that the $\nu_1$-orbits of the simple objects in $\mo_0 A$ are all finite. Minimality of $\vec{d}$ shows that every indecomposable summand of $M$ must be simple in $\mo_0 A$ so $M$ is regular semisimple. Finally, \cite[Section~2 (S8)(iii)]{LdP} ensures that up to isomorphism, all but finitely many $M$ are in fact regular simple.
\end{proof}

\section{The tautological moduli problem}\label{sec:taut}

 In this section, we introduce the class of {\em smoothly weighted varieties} which generalises the notion of weighted projective curves. We show that for such stacks, they can be recovered as a tautological Serre stable moduli stack of sheaves. This explains why the Serre stable moduli stack is useful, whenever we have an abelian category derived equivalent to such a stack. 

A stack $\Y$ is called a {\em smoothly weighted (projective) variety} if there is a morphism $\pi: \Y \lm Y$ to a (projective) variety $Y$ such that every point $p \in Y$ has an open neighbourhood $U$ and the restricted map $\pi_{\pi^{-1}(U)}: \pi^{-1}(U) \lm U$ has the form $[\tilde{U}/G] \lm U$ for some   1-generated cyclic cover $\tilde{U}/U$ (defined Section~\ref{sec:serrestacks}) with Galois group $G$ and $\tilde{U}$ smooth. Note smoothness of $\tilde{U}$ implies smoothness of the ramification locus. Also, $\pi$ is an isomorphism away from the ramification loci of the covers $\tilde{U}/U$. The coarse moduli scheme of $\Y$ is $Y$. By construction, weighted projective curves are examples of smoothly weighted projective varieties. The simplest examples of Geigle-Lenzing projective spaces \cite{HIMO} as described below are another. 

\begin{eg}[Geigle-Lenzing projective space]\label{eg:GLspace}
Consider the graded polynomial algebra $R = k[x_0,\ldots,x_n]$ where $\deg x_0 = 1$ but $\deg x_i= p$ for $i>0$ and some integer $p > 1$. We may consider the stacky proj $\Y$ of this which can be described as follows. The $\Z$-grading on $R$ amounts to an action of  $k^{\times}$ on $\mathbb{A}^{n+1}$ and we define the {\em GL-projective space weighted on a hyperplane with weight $p$} to be $\Y = [\left(\mathbb{A}^n - 0\right)/k^{\times}]$. There is a map to the scheme-theoretic proj, $\proj R \simeq \proj k[t_0 = x_0^p,x_1,\ldots,x_n] = \PP^n$.  As in the appendix, 
we have a local description of this map $\Y \lm \PP^n$. Above the affine patch $t_0 \neq 0$, $\Y$ is just $\spec R/(x_0 - 1) \simeq \mathbb{A}^n$. Above the affine patch $x_i\neq0,  i>0$, it is $[\left(\spec R/(x_i -1)\right)/\mu_p]$ so the inertia groups along $t_0 = 0$ are $p$ and $\Y$ is a smoothly weighted projective variety. 
\end{eg}

In the theory of moduli of sheaves, we need to fix a discrete invariant, which is trickier than for $A$-modules because $K_0(\Y)$ may no longer be discrete. If $\Y$ is a projective scheme, the usual approach is to polarise the scheme with a very ample line bundle and use the Hilbert polynomial as the discrete invariant. If the coarse moduli scheme $Y$ for $\Y$ is projective, we can do the same as then $\coh \Y \simeq \cala \mo$ for some order $\cala$ on $Y$. However, the discrete invariant we are interested in is the one corresponding to any skyscraper sheaf $k(p)$ where $p$ is a point on the scheme locus of $\Y$. For this special case, the {\em moduli stack $\tilde{\W}$ of skyscraper sheaves} can be defined directly as follows. Given a test scheme $T$, the category of $T$-points $\tilde{\W}(T)$ consists of sheaves $\calm$ on $\Y \times T$ which are flat over $T$ and furthermore, locally free of rank 1 over $T$, that is, $\text{Supp}\ \calm$ is finite over $T$ and its push forward to $T$ is a line bundle. As in Subsection~\ref{subsec:rigid}, we will rigidify this stack to obtain the rigidified moduli stack $\W$. 

On an open patch of form $\mathbb{U} = [\tilde{U}/G]\subseteq \Y$ we may identify $\coh \Y$ with $\cala \mod$ where $\cala = \calo_{\tilde{U}} \# G$. Now $\cala$ is an order over $U$ which is Azumaya away from the ramification locus. Above an unramified point $p \in Y$, the skyscraper sheaf $k(p)$ on $\Y$ corresponds to the unique simple $\cala$-module $M_p$ supported at $p$. This module has dimension $|G|$ over $k(p)$ and in fact is naturally a $k(p) G$-module which is isomorphic to $k(p)G$. In other words, each irreducible character occurs with multiplicity 1 in $M_p$. Suppose now that $R$ is a commutative ring and $\calm \in \W(R)$, viewed as a flat family of $\cala$-modules over $R$. If $\eps_{\mu}\in kG$ is the primitive idempotent corresponding to the character $\mu \in G^{\vee}$, then $\eps_{\mu} \calm$ is a locally free $R$-module of rank 1. 

We fix the shift parameter to be $d = \dim \Y$ which in our case is just $\dim Y$. 
Recall from Section~\ref{sec:serrestacks} that we have a (shifted) Serre functor $\nu_d := \w_{\Y}
\otimes_{\Y} (-)$. This definition naturally generalises to families for given a flat family $\calm
\in \W(R)$, $\nu_d(\calm) = \w_{\Y} \otimes_{\Y} \calm$ defines a flat family of skyscraper sheaves on $\Y$. Since we have assumed that $\Y$ is smooth, $\nu_d$ defines an automorphism on the whole of $\W$ and we do not need to worry about its domain of definition as occurs for moduli of $A$-modules. We can now define Serre stable sheaves as before, as well as the Serre stable moduli stack $\W^S$. More precisely, $\W^S$ is the stackification of the 
pre-stack $\W^{S,pre}$ whose objects over a test scheme $T$ are flat families $\calm$ of skyscraper sheaves on $\Y$
over $T$, together with an isomorphism $ \nu_d \calm \xrightarrow{\sim} \calm \otimes_T \caln$ 
such that $\calm,\caln$ are line bundles over $T$. 

Before proving the next result, it will be useful to keep in mind that if $\Y = Y$ is actually a variety, then $Y \simeq \W$ and the isomorphism is given by the universal skyscraper sheaf $\calo_{\Delta} \in \coh Y \times Y$ where $\Delta \subset Y \times Y$ is the diagonal copy of $Y$. If $Y = \spec \calo$ is affine, then the inverse map $\Psi$ is easily defined as follows. Consider a flat family $\calm \in \W(R)$ over $R$ which we view as an $(\calo, R)$-bimodule which is locally free rank 1 over $R$. Left multiplication induces a ring homomorphism $\calo \lm \End_R \calm = R$ and hence an $R$-point $\spec R \lm Y$ which we define to be $\Psi(\calm)$. The proof below generalises this argument. 

\begin{lemma}\label{lemma:order}
Let $\tilde{U}/U$ be a $\chi$-generated $G$-cover of a smooth affine variety ramified over a smooth divisor and $\Y = [\tilde{U}/G]$ be the corresponding quotient stack. Then we have an isomorphism $\W^S \simeq \Y$. 
\end{lemma}
\begin{proof}
We fix the usual notation: $U = \spec \calo, \tilde{U} = \spec
\calot$ and $\cala = \calot \# G$ is the skew group ring so there is a category equivalence $\coh\Y\simeq \cala\mod$. We will view a flat family of $\cala$-modules over a commutative ring $R$ as an $(\cala,R)$-bimodule to help us keep track of scalars. In particular, if $\eps_{\mu}\in kG$ is the primitive idempotent corresponding to $\mu \in G^{\vee} = \langle \chi \rangle$ and $\calm \in \W(R)$, then $\calm$ is a left $\cala$-module such that $\eps_{\mu}\calm$ is a locally free right $R$-module of rank 1.

We define a morphism of stacks $\Phi\colon \Y \to \W^S$ by defining the universal family. First note that $\mathcal{U} =
_{\cala} \cala_{\calot}$ is a family of left $\cala$-modules over $\calot$ and furthermore, $\eps_{\mu} \mathcal{U} \simeq \calot$ as a right $\calot$-module. This flat family is also
$G$-equivariant since $\mathcal{U}$ is an $(\cala, \calot \# G)$-bimodule. Hence $\mathcal{U}$ defines a $G$-equivariant element of $\W(\calot)$, that is a flat family of $\cala$-modules over $\Y$. It corresponds to $\calo_{\Delta}$ above. Note $\mathcal{U}$ is also Serre stable. Indeed we see from Lemma~\ref{lem:omega} that there is a natural isomorphism of $(\cala,\calot)$-bimodules 
$$\w_{\cala} \otimes_{\cala} \mathcal{U} = \w_{\cala} \simeq
\cala \otimes_{\calot} \tilde{\w} = \mathcal{U} \otimes_{\calot} \tilde{\w}$$
We have thus defined an element of $\W^S(\calot)$. Moreover, all data are $G$-equivariant so we do indeed
have a morphism $\Y = [\spec \calot/G] \to \W^S$. 


We now construct the inverse functor $\Psi$ to $\Phi$ and by stackification, it suffices to define
$\Psi\colon \W^{S,pre} \to \Y$. Consider an object of $\W^{S,pre}(R)$ given by an $(\cala,R)$-bimodule $
_{\cala}\calm_R$ and $(\cala,R)$-bimodule isomorphism $\theta: \w_{\cala} \otimes_{\cala} \calm \xrightarrow{\sim} \calm \otimes_R \caln$ where $\caln$ is an invertible $R$-module. Since $\cala = \calot \# G$, we may view the $\cala$-module structure of $\calm$ as a $G^{\vee}$-grading 
$$\calm = \bigoplus_{\lambda \in G^{\vee}}  \eps_{\lambda} \calm $$
together with a compatible action of $\calot$, that is,  a $G^{\vee}$-graded algebra homomorphism $\calot \lm \End_R \calm$. This follows from the more precise formula in $\calot \# G$: given  a degree $\lambda$ element $s$ of $\calot$ and $\mu \in G^{\vee}$ we have $s \eps_{\mu} = \eps_{\mu + \lambda} s$. 

The existence of the isomorphism $\theta$ imposes the following structure on $\calm$.

\begin{sublemma}  \label{sublem:calm}
\begin{enumerate}
\item There is a homomorphism $\rho: \calo \lm R$ such that left multiplication by $s \in \calo$ on $\calm$ is right multiplication by $\rho(s)$. 
\item There is an invertible $R$-module $\call$ such that $\theta$ induces an isomorphism $\theta_{\mu}: \eps_{\mu + \chi} \calm \xlm{\sim} \eps_{\mu} \calm \otimes_R \call$.
\item $\theta$ also induces an isomorphism $\call^n \simeq R$.
\end{enumerate}
\end{sublemma}
\begin{proof}
Let $\tilde{\w}_{\chi}$ be the eigensheaf of $\tilde{\w}$ corresponding to the character $\chi$ so that by Lemma~\ref{lem:omega}, $\theta$ can be a viewed as an isomorphism $\theta: \tilde{\w}_{\chi} \otimes_{\calo} \calm \xlm{\sim} \calm \otimes_R \caln$. Multiplying by $\eps_{\mu}$ and re-arranging gives an isomorphism
\begin{equation}
\eps_{\mu + \chi} \calm \simeq \tilde{\w}_{\chi}^{-1} \otimes_{\calo} \eps_{\mu}\calm \otimes_R \caln.
\label{epsi}
\end{equation} 
For each $\mu \in G^{\vee}$, left multiplication induces a ring homomorphism $\rho_{\mu}: \calo \lm \End_R \eps_{\mu} \calm = R$. Since (\ref{epsi}) is an  $(\calo,R)$-bimodule isomorphism, we see $\rho_{\mu}$ is independent of the choice of $\mu$ and part~i) follows. We may thus rewrite (\ref{epsi}) as 
$$\eps_{\mu + \chi} \calm \simeq \eps_{\mu}\calm \otimes_R \caln \otimes_R \mathcal{P}.$$
where $\mathcal{P} = \tilde{\w}_{\chi}^{-1} \otimes_{\calo} R$. Now $\chi$ has order $n$ so the invertible module $\caln \otimes_R \mathcal{P}$ is $n$-torsion. Setting $\call = \caln \otimes_R \mathcal{P}$ finishes the proof of the sublemma. 
\end{proof}
We wish now to define $\Psi(\calm)$ which consists of an \'etale $G$-cover $\tilde{R}$ of $R$ and a $G$-equivariant homomorphism $\calot \lm \tilde{R}$. Needless to say, $\tilde{R}$ will be given by the $n$-torsion line bundle $\call$ in the sublemma, but to make this functorial, we proceed as follows. Note first that by Sublemma~\ref{sublem:calm}i), any right $R$-linear endomorphism of $\calm$ is automatically left $\calo$-linear too.  We let $\tilde{R}$ be the $R$-subalgebra of $\End_R \calm$ consisting of endomorphisms $\phi$ compatible with $\theta$, that is, $(\phi \otimes \id_{\caln}) \theta = \theta
(\id_{\tilde{\w}_{\chi}} \otimes \phi)$ as morphisms from $\tilde{\w}_{\chi} \otimes_{\calo} \calm \lm \calm \otimes_R \caln$. Note first that the construction of $\tilde R$ does not depend on the choice of isomorphism $\theta$ in its equivalence class. Now $\tilde{\w}$ is an $\calot$-central bimodule containing $\tilde{\w}_{\chi}$ so left multiplication by elements of $\calot$ is certainly compatible with $\theta$. We thus obtain a $G^{\vee}$-graded homomorphism $\calot \lm \tilde{R}$ or equivalently, a $G$-equivariant homomorphism. 

We now show that $\tilde{R}$ is isomorphic to the commutative  $G^{\vee}$-graded $R$-algebra $\oplus_{i = 0}^{n-1} \call^i$ and so is indeed an \'etale $G$-cover of $R$, thus completing the definition of $\Psi$. First note that a degree $\lambda$ endomorphism $\phi \in \End_R \calm$ is compatible with $\theta$ if the following diagram commutes for all $\mu \in G^{\vee}$
$$
\begin{CD}
 \eps_{\mu + \chi} \calm @>{\theta_{\mu}}>> \eps_{\mu} \calm \otimes_R \call \\
@V{\phi}VV @V{\phi \otimes \id}VV \\
 \eps_{\mu +\lambda +  \chi} \calm @>{\theta_{\mu + \lambda}}>> \eps_{\mu + \lambda} \calm \otimes_R \call
\end{CD}$$
We now easily see that multiplication by elements of $\call$ via $\theta_{\mu}^{-1}$ are endomorphisms of $\calm$ compatible with $\theta$. There is hence a $G^{\vee}$-graded $R$-algebra homomorphism $\oplus_{i = 0}^{n-1} \call^i \lm \tilde{R}$. 
This homomorphism is clearly injective and is also surjective since the commutative diagram above shows that any homogeneous $r \in \tilde{R}$ is completely determined by its restriction to $\eps_0 \calm$ which is an element of $\Hom_R(\eps_0 \calm, \eps_{i\chi} \calm) \simeq \call^i$ where $\deg r = i\chi$. This completes the definition of $\Psi(\calm)$ which is clearly seen to be functorial in $\calm$. 

The construction of $\Psi(\calm)$ is completely reversible so $\Psi$ does give an equivalence of stacks. We will illustrate this for the ``tautological point'' $\pi: \tilde{U} \lm \Y$ and so see that the inverse of $\Psi$ is indeed given by the morphism $\Phi: \Y \lm \W^S$ induced by the universal family $\mathcal{U} = _{\cala}\cala_{\calot}$ described above. 

To this end, first note that the tautological point is given by the trivial $G$-torsor $pr: G
\times \tilde{U} \lm \tilde{U}$ and $G$-equivariant map $\alpha: G \times \tilde{U} \lm \tilde{U}$
given by the $G$-action. We determine $\calm:= \Psi^{-1}(\pi)$. Now the trivial $G$-torsor is
given by the trivial line bundle $\call = \calot$ so we have $\caln \otimes_{\calo}
\tilde{\w}_{\chi}^{-1} = \calot$ or equivalently by Lemma~\ref{lem:omega}, $\caln = \tilde{\w}_{\chi}
\otimes_{\calo} \calot= \tilde{\w}$. Since we are working in the rigidified moduli stack $\W$, we
may assume that $\eps_0 \calm = \calot$.  Reversing equation (\ref{epsi}), we define $\eps_{i\chi}
\calm = \tilde{\w}_{\chi}^{-i} \otimes_{\calo} \calot \otimes_{\calot} \tilde{\w}^i \simeq \calot$.
Setting $\calm = \oplus_i \eps_{i\chi} \calm = \oplus_{\mu \in \G^{\vee}} \eps_{\mu} \calot$, we may
sum the isomorphisms in (\ref{epsi}) to give an isomorphism $\theta: \tilde{\w}_{\chi} \otimes_{\calo}
\calm  \simeq \calm \otimes_{\calot} \tilde{\w}$. Note that $\calm$ is defined as a
$G^{\vee}$-graded right $\calot$-module or equivalently, a $(kG,\calot)$-bimodule which is naturally
isomorphic to $\mathcal{U} = \calot \# G$. To complete the left $\cala$-module structure, we need a
compatible left $\calot$-module structure which is provided by the $G^{\vee}$-graded homomorphism $\alpha^*: \calot \lm \calo_{G \times \tilde{U}} = \calot G^{\vee}$ where $\calot G^{\vee} = \oplus_{\mu \in G^{\vee}} \calot u_{\mu}$ is the usual group ring on the dual group $G^{\vee}$ (so $u_{\mu}u_{\lambda} = u_{\mu + \lambda}$). The endomorphisms of $\calm$ which are compatible with $\theta$ can be identified with $\calot G^{\vee}$ and under this identification, $u_{\mu}$ permutes the graded components by $\eps_{\lambda} \mapsto \eps_{\mu + \lambda}$. We view $\calot$ as a $kG$-module so  for $f \in \calot$ and writing ``.'' for scalar multiplication to avoid confusion, we have $\eps_{\mu}. f$ is the projection of $f$ onto the $\mu$-eigensheaf of $\calot$. In this language, $\alpha^* f = \sum_{\mu\in G^{\vee}} (\eps_{\mu}.f) u_{\mu}$. It follows that if $f$ is homogeneous of degree $\mu$, then $\alpha^* f = f u_{\mu}$ so left multiplication by $f$ induces the map $\eps_{\lambda} \calot \lm \eps_{\mu + \lambda} \calm: \eps_{\lambda} g \mapsto \eps_{\mu + \lambda} gf$. Hence $\calm \simeq \mathcal{U}$ as $(\cala,\calot)$-bimodules. This completes the proof that $\Psi^{-1} = \Phi$. 
\end{proof}

\begin{thm}\label{thm:order}
Let $\Y$ be a smoothly weighted variety. Then there is a natural isomorphism $\W^S\simeq \Y$.
\end{thm}
\begin{proof}
This follows from patching the isomorphisms in the lemma. The easiest way to see that the isomorphisms do indeed glue together is to note that $\Phi: \Y \lm \W^S$ is defined by the universal family $\mathcal{U} = \calo_{\Delta}$ and natural isomorphism $\w_{\Y} \otimes_{\Y} \calo_{\Delta} \simeq \calo_{\Delta} \otimes_{\Y} \w_{\Y}$. 
\end{proof}

Suppose $Y$ is the coarse moduli scheme of the smoothly weighted variety $\Y$. If $\Y$ is weighted along some divisor $D \subset Y$ with weight $p > 1$, then every point $y \in D$ corresponds to a $k$-point of $\W^S$ of the form $\mathcal{F}  \oplus  \nu_d \mathcal{F} \oplus \ldots \oplus \nu_d^{p-1} \mathcal{F}$ for some simple sheaf $\mathcal{F}$. Viewing this as a point of $\W$, we see its rigidified automorphism group in $\W$ is $(k^{\times})^{p-1}$, so there is no chance that $\W$ recovers the stack $\Y$. Similarly, if in our definition of smoothly weighted varieties, we had allowed weighting along intersecting divisors, one would find that the rigidified automorphism groups in $\W^S$ are sometimes infinite, so $\W^S$ will not recover $\Y$ in these more general cases. We suspect that there are ``higher'' versions of the Serre stable moduli stack which will work, perhaps involving the cotangent bundle instead of just the canonical bundle.  

\section{The dual of the tilting bundle is universal}\label{sec:universal}

Let $\Y$ be a smoothly weighted projective variety of dimension $d$ and $\calt= \oplus_{v \in Q_0} \calt_v$ be a tilting bundle on $\Y$ where the $\calt_v$ are the indecomposable summands. We will assume $\calt$ is {\em basic} whereby we mean the $\calt_v$ are non-isomorphic so the endomorphism algebra $A = \End_{\Y} \calt$ is also basic and the quiver corresponding to $A$ has $Q_0$ as its vertices. We consider the dimension vector $d: Q_0 \lm \mathbb{N}: v \mapsto \rk \calt_v$ so that the generic skyscraper sheaf on $\Y$ corresponds to an $A$-module of dimension vector $\vec{d}$. We may construct the rigidified moduli stack $\X$ of $A$-modules with dimension vector $\vec{d}$. We fix the shift parameter to be $d$. The aim of this section is to prove that, at least in the (anti-)Fano case defined below, the Serre stable moduli stack $\X^S$ is isomorphic to $\Y$ and that the universal module is given by $\calt^{\vee}$. 

We first recall from \cite[Section~5]{K2} how to calculate the derived tensor product $\mathcal{F} \otimes^\L_A M$ of an $(\calo_{\Y},A)$-bimodule $\mathcal{F}$ with a left $A$-module $M$. We can take a projective $A$-bimodule resolution of $A$ which has the form $P_{\bullet}$ where each $P_q$ is a finite direct sum of bimodules of the form $Ae_v \otimes_k e_w A$ for various $v,w \in Q_0$. Then 
\begin{equation}
\mathcal{F} \otimes^\L_A M = \mathcal{F} \otimes_A P_{\bullet} \otimes_A M .
\label{eq:derivedtensor}
\end{equation}
Indeed, $\mathcal{F} \otimes_A Ae_v \otimes_k e_w A = \mathcal{F}_v \otimes_k e_w A$, so $\mathcal{F} \otimes_A P_{\bullet} $ can be taken to be a projective resolution of $\mathcal{F}$ as an $A$-module. Of course, a similar statement can be made about $P_{\bullet} \otimes_A M$. We let $\D^+(\Y,A)$ denote the bounded below derived category of quasi-coherent $(\calo_{\Y},A)$-bimodules. Below, we use the derived global sections functor $\mathbf{R}\Gamma: \D^+(\Y)  \lm \D^+(k)$ on $\Y$. We need a preliminary lemma.

\begin{lemma}\label{lem:tensorh0commute}
Let $\mathcal{F} \in \D^+(\Y,A)$ and $M \in \D^b_{fg}(A^{op})$. Then 
$$\mathbf{R}\Gamma (\mathcal{F}) \otimes^\L_A M \simeq 
\mathbf{R}\Gamma (\mathcal{F} \otimes^\L_A M).$$
\end{lemma}
\begin{proof}
We may replace $\mathcal{F}$ with a bounded below complex of injective $(\calo_{\Y},A)$-bimodules and $M$ with a bounded complex of finite projective left $A$-modules. Since taking global sections of an $(\calo_{\Y},A)$-bimodule commutes with $- \otimes_A Ae_v$, the lemma follows. 
\end{proof}
We now show that $\calt,\calt^{\vee}$ are Serre stable left (respectively right) $A$-modules.

\begin{prop}\label{prop:tiltstable}
There are natural isomorphisms of bimodules \[  DA[-1] \otimes^{\L}_A \calt \simeq
\calt\otimes_{\Y}\omega_{\Y} ,\quad 
\calt^{\vee} \otimes^{\L}_A  DA[-1]\simeq 
 \omega_{\Y} \otimes_{\Y}\calt^\vee.\] In particular, $\calt^{\vee}$ is Serre stable.
\end{prop}
\begin{proof}
	Note that $-\otimes_A^{\L}\calt\colon\D^b(A)\to\D^b(\Y)$ is a category equivalence which
	commutes with the Serre functor and so:\[-\otimes_A^{\L}DA[-1]\otimes_A^{\L}\calt\simeq
	-\otimes_A^{\L}\calt\otimes_{\Y}\omega_{\Y}\] from which the first isomorphism follows.
Similarly, $\Rhom_{\Y}(\calt,-)$ commutes with the Serre functor and, using Lemma \ref{lem:tensorh0commute} we see there are natural isomorphisms
\[\RG(-\otimes^\L_{\Y} \omega_{\Y} \otimes^\L_{\Y} \calt^{\vee})  \simeq
\RG(- \otimes^\L_{\Y} \calt^{\vee}) \otimes^\L_A DA[-1] \simeq 
\RG(- \otimes^\L_{\Y} \calt^{\vee} \otimes^\L_A DA[-1])
\]
Applying this functor to $\calo_U$ where $U \subset \Y$ varies over open subsets of $\Y$ show that there is an isomorphism of $(\calo_{\Y},A)$-bimodules $\omega_{\Y} \otimes^\L_{\Y} \calt^{\vee} \simeq \calt^{\vee} \otimes^\L_A DA[-1]$.
\end{proof}

We wish to show that the derived equivalence $- \otimes^\L_A \calt$ sends flat families of Serre stable $A$-modules of dimension vector $\vec{d}$ to flat families of skyscraper sheaves on $\Y$. This will allow us to use Theorem~\ref{thm:order}. 
Our first task is to show that $M \otimes^{\L}_A \calt$ is a sheaf for any $k$-point $M$ of $\X^S$. As in \cite{BO}, the theory of point objects is most useful when $\Y$ is {\em (anti-)Fano} whereby we mean that $\omega_{\Y}$ (respectively $\omega_{\Y}^{-1})$ is ample. Below, we let $Y$ denote the coarse moduli space of $\Y$ and the support of coherent sheaves on $\Y$ will mean their supports as closed subsets of $Y$. 

\begin{lemma}\label{lem:heartfano}
Suppose that $\Y$ is Fano or anti-Fano. Then $M \otimes^{\L}_A \calt$ is a finite length sheaf for any $k$-point $M$ of $\X^S$.
\end{lemma}
\begin{proof}
Note that $A$ has finite global dimension so Serre stability of $M$ means that $\Ext^i_A(M,A) = D\Ext^{d-i}_A(A,M)$ and hence $\pd_A M = d$. Choosing a length $d$ projective resolution for $M$, we may assume that $F = M \otimes^{\L}_A \calt$ is a length $d$ complex 
$$ 0 \lm F^{-d} \xrightarrow{\delta^{-d+1}} F^{-d+1} \lm \ldots \lm F^0 \lm 0 $$
of locally free sheaves on $\Y$.  From Proposition~\ref{prop:tiltstable}, its cohomologies $H^{-j}$ are Serre stable and hence have finite support since we are assuming that $\Y$ is either Fano or anti-Fano. We will show by downward induction on $j$ that $H^{-j} = 0$ for $j>0$. Since the cohomologies are finite length sheaves, we may work locally on $Y$, in which case the sheaves $F^{-j}$ can be viewed as $\cala$-modules where $\cala$ is an $\oy$-order as in Section~\ref{sec:serrestacks} and $\oy$ is a local ring. Assuming cohomologies vanish in degrees $<-j$, then the $\oy$-module $C^{-j} = \coker\left(\delta^{-j}:F^{-j-1} \lm F^{-j}\right)$ has projective dimension $\leq d-j$. By Auslander-Buchsbaum, $\text{depth}_R C^{-j} \geq j$ so $C^{-j}$ has no non-zero finite length submodules unless $j = 0$. 
\end{proof}

\begin{lemma}\label{lem:heartwtdline}
Suppose that $\Y$ is a weighted projective line and $\vec{d}$ is minimal Coxeter stable. Then $M \otimes^{\L}_A \calt$ is a finite length sheaf for any $k$-point $M$ of $\X^S$.
\end{lemma}
\begin{proof}
In this case, $\pd M = 1$ and $\coh \Y$ is hereditary so $M \otimes^{\L}_A \calt \simeq F_1[1] \oplus F_0$ for Serre stable sheaves $F_1,F_0$. Hence $M$ is the direct sum of the Serre stable modules $\Rhom_{\Y}({\calt}, F_0), \Rhom_{\Y}({\calt}, F_1[1])$. Minimality of $\vec{d}$ then ensures that either $F_0= 0$ or $F_1 = 0$. Now rank considerations show that for $F_1[1] \oplus F_0$ to have the same class in $K_0(\Y)$ as a skyscraper sheaf, we must have $F_1 = 0$ and $F_0$ is a finite length sheaf. 
\end{proof}

\begin{prop}\label{prop:stillflat}
Suppose that $M \otimes^{\L}_A \calt$ is a finite length sheaf for any $k$-point $M$ of $\X^S$. 
Let $\calm$ be a flat family of Serre stable $A$-modules over a noetherian ring $R$ with dimension vector $\vec{d}$
(so there is an isomorphism $\calm \simeq \call \otimes_R \nu_d \calm$ for some line bundle $\call$
on $\spec R$). Then $\calm \otimes^\L_A \calt$ is a flat family of $\calo_{\Y}$-modules over
$R$.
\end{prop}
\begin{proof}
We may assume that $R$ is a local ring with residue field $\kappa$. 
We let $K_{\bullet} = \calm \otimes^{\L}_{A} \calt$ and use the hypertor spectral sequence for $\kappa \otimes^{\L}_R K_{\bullet}$. Note that since
$\calm$ is locally free over $R$ we have 
$\kappa\otimes_R^{\L}K_{\bullet}=\kappa\otimes_RK_{\bullet}$ and by assumption $h_i(\kappa\otimes_RK_{\bullet})=0$ if $i\neq 0$. The homology 
spectral sequence thus becomes \[E_{i,j}^2=\Tor^R_i(\kappa, h_{j}(K_\bullet))\Rightarrow h_{i+j}(\kappa \otimes_R K_\bullet)\] with the $E^{\infty}$ term on 
the right hand side vanishing when $i+j \neq 0$.  Consider the
support $Z \subset \Y_R$ of $\calm \otimes_A \calt$ in $\Y_R$. The composite $Z \hookrightarrow \Y_R
\xrightarrow{\pi} \spec R$ is both quasi-finite and projective so by Zariski's main theorem, the
morphism is finite. Hence $\calm \otimes_A \calt=h_0(K_\bullet)$ is a finitely generated $R$-module.
From the spectral sequence we see that $\Tor^R_1(\kappa,h_0(K_\bullet))=0$ and so from the local
criterion for flatness we deduce that $h_0(K_\bullet) = \calm\otimes_A\calt$ is flat over
$R$. Hence $\Tor_i^R(\kappa,h_0( K_\bullet))=0$ for all $i>0$. Now, the spectral sequence shows that
$\Tor_0^R(\kappa,h_1( K_\bullet))= \kappa \otimes_R \h_1(K_\bullet)=0$. However, just as in the
$h_0$ case $h_1(K_\bullet)$ is a finitely generated $R$-module and so $h_1(K_\bullet)=0$ and in particular
$E_{i,1}^2=0$ for all $i$. Continuing by induction we see that $h_i(K_\bullet)=0$ for all $i\ne 0$.

\end{proof}

\begin{thm}\label{thm:universal}
Suppose $\Y$ is either i) Fano or anti-Fano or, ii) that it is a weighted projective line and $\vec{d}$ is minimal Coxeter stable. Then the flat family of Serre stable $A$-modules $\calt^{\vee} \simeq \omega^{-1}_{\Y} \otimes_{\Y} \nu_d\calt^{\vee}$ defines an isomorphism $\Y \to \X^S$ of stacks. 
\end{thm}
\noindent
\textbf{Remark} Together with Theorem~\ref{tBKR} and Proposition~\ref{p:regssforcc}, part (i) gives the characterisation of non-tubular concealed canonical algebras stated in Theorem~\ref{thm:3}. 
\begin{proof}
We must show that $\calt^{\vee} \simeq \omega^{-1}_{\Y} \otimes_{\Y} \nu_d \calt^{\vee}$ gives the
universal family.  We consider a flat family of $A$-modules $\calm$ over $R$ with dimension vector
$\vec{d}$ and an isomorphism $\theta\colon\calm \to \call \otimes_R \nu_d \calm$. From Proposition~\ref{prop:stillflat}, we know that $\calm \otimes^\L_A \calt = \calm \otimes_A \calt$ is a flat family of $\calo_{\Y}$-modules. Furthermore, 
$\theta$ and the natural isomorphism of Proposition~\ref{prop:tiltstable} give isomorphisms
\begin{equation}\calm \otimes_A \calt \simeq \call \otimes_R \calm \otimes^\L_A DA[-1]  \otimes^\L_A \calt \simeq \call \otimes_R \calm \otimes^\L_A \calt \otimes_{\Y} \omega_{\Y} 
\label{eq:stable}.
\end{equation}
In particular $\calm \otimes_A \calt$ is Serre stable and Theorem~\ref{thm:order} shows that there is a morphism $f:\spec R \lm \Y$ such that $\calm \otimes_A \calt \simeq f^* \calo_{\Delta}$ where $\Delta \subset \Y \times \Y$ is the diagonal. In fact the isomorphism above is obtained from pulling back $ \calo_{\Delta} \simeq \omega^{-1}_{\Y} \otimes_{\Y}\calo_{\Delta} \otimes_A \omega_{\Y}$ via $f$ where here, we view $\calo_{\Delta}$ as a $(\calo_{\Y},\calo_{\Y})$-bimodule and the pullback when viewed as a sheaf on $\Y \times \Y$ is via $f \times 1$. Since $\calt$ is a tilting bundle we have 
\begin{eqnarray*}
\calm & \simeq & \calm \otimes^{\L}_A \RG( \calt \otimes_{\Y} \calt^{\vee}) \\
 & \simeq & \RG(\calm \otimes_A \calt \otimes_{\Y} \calt^{\vee})\\
 & \simeq & \RG(f^*\calo_{\Delta} \otimes_{\Y} \calt^{\vee}) \\
 & \simeq & \RG(f^*\calt^{\vee}) \\
 & \simeq & f^* \calt^{\vee}
\end{eqnarray*} 
where the second isomorphism follows from lemma~\ref{lem:tensorh0commute}. This calculation can be used to show that applying $\RG(- \otimes_{\Y} \calt^{\vee})$ to the isomorphism in (\ref{eq:stable}) recovers $\theta$ and that $\calt^{\vee} \simeq \omega^{-1} \otimes_{\Y} \nu_1 \calt^{\vee}$ is indeed universal. 
\end{proof}

Recall from Proposition~\ref{poneismin} that the theorem applies in the case where $\Y$ is a weighted projective line and $\calt$ is the canonical tilting bundle whose endomorphism ring is the canonical algebra. It also applies to the $n$-canonical algebras (see \cite[Section~6]{HIMO}) for a GL-projective space weighted on a hyperplane as defined in Example~\ref{eg:GLspace}. We suspect that the hypotheses of the theorem can be weakened significantly since i) and ii) above are quite independent. In particular, we hope the theorem holds true for all concealed canonical algebras. 

\section{Appendix: Classical approach to weighted projective lines} \label{sec:append}

In this appendix, we clarify the relationship between Geigle-Lenzing description of weighted projective lines in \cite{GL} and ours. The material is implicit in \cite{GL}. 

Let $G$ be a commutative reductive algebraic group, by which we
simply mean one isomorphic to $\mathbb{G}_m^r \times A$ for some
finite abelian group $A$. Let $\Gamma = G^{\vee}$ be the dual (or
character) group, which is a finitely generated abelian group of
rank $r$. We first recall that to give a rational action of $G$ on
an affine scheme $\spec R$ amounts to imposing a $\Gamma$-grading
on $R$. Given $\gamma \in \Gamma$, $G$ acts on the $\gamma$-graded component $R_{\gamma}$ via
the character $\gamma$. Suppose we are indeed given such a grading
$R = \oplus_{\gamma} R_{\gamma}$. In this language, a
$G$-equivariant quasi-coherent sheaf on $\spec R$ corresponds to a
graded $R$-module.  More precisely, we have a category equivalence
$\qcoh [\spec R /G] \simeq \Gr R$. 

We now restrict to the case where $\Gamma$ has rank 1, though the
ideas here apply in general. In this case, there are precisely two
surjective maps $\nu: \Gamma \lm \Z$. We will assume that the
$\Gamma$-graded $k$-algebra $R$ is {\em connected} in the sense
that $R_0 = k$ and we can choose $\nu$ so that for any $\gamma \in
\Gamma$ with $R_{\gamma} \neq 0$, we have $\nu (\gamma) >0$. This
ensures that $\mmm := \oplus_{\gamma \neq 0} R_{\gamma}$ is a
graded ideal in $R$ and we let $\text{pt} \in \spec R$ denote that
corresponding closed point. Since $\text{pt}$
is fixed by $G$, $G$ acts on the open set $U = \spec R -
\text{pt}$, and we may consider the stacky projective scheme
$\text{StProj}\ R = [U/G]$. By \cite[Example~7.21]{V}, we know that there is a category equivalence $\qcoh \text{StProj}\ R \simeq
(\Gr R)/\text{tors}$ where $\text{tors}$ is the Serre subcategory of
$\mmm$-torsion modules.   Note that the category of quasi-coherent sheaves on a weighted projective line as described in \cite{GL}  has the form $(\Gr R)/\text{tors}$ for an appropriate choice of 2-dimensional graded ring $R$. Hence we may identify the weighted projective lines of Geigle-Lenzing with the associated
stack $\text{StProj}\ R$.

We now analyse the stack $\mathbb{X} = \text{StProj}\ R$ in analogy with the standard construction of projective schemes by patching affine
open sets. This will connect the Geigle-Lenzing approach with the one given in Subsection~\ref{subsec:wtdlines}. For each non-zero homogeneous element $t \in
R_{\gamma}, \gamma \neq 0$, the set $U_t = \spec R[t^{-1}]$ is a
$G$-invariant open subset of $U$ and hence $[U_t/G]$ is an open
substack of $\mathbb{X}$. These cover $\mathbb{X}$ and $[U_t/G]$
has a coarse moduli scheme $\spec (R[t^{-1}])^G = \spec
R[t^{-1}]_0$. Hence $\mathbb{X}$ has a coarse moduli scheme which
is the usual scheme-theoretic $\proj R$. 

The open substack $[U_t/G]$ has a simpler description which makes
it obvious that it is a Deligne-Mumford stack and what are the
inertia groups. We change the presentation for the stack by
replacing $U_t$ with the closed subscheme $\bar{U}_t = \spec
R[t^{-1}]/(t-1) = \spec R/(t-1)$.  First consider the exact
sequence defining $\bar{\Gamma}$ below $$ 0 \lm \Z \gamma \lm
\Gamma \lm \bar{\Gamma} \lm 0 $$ and the corresponding dual exact
sequence $$ 1 \lm G' \lm G \xrightarrow{\gamma} k^{\times} \lm 1$$
where $G' \leq G$ is the subgroup isomorphic to
$\bar{\Gamma}^{\vee}$ and we have identified $(\Z\gamma)^{\vee}$
with $k^{\times}$. This sequence shows that $G'$ fixes $t-1$ so
acts on $\bar{U}_t$. We leave the reader to verify that $[U_t/G]
\simeq [\bar{U}_t/G']$, contenting ourselves with heuristics
(which form a proof when $U_t$ is a variety) and a proof that they have equivalent categories of quasi-coherent sheaves (Proposition~\ref{pstackred} below). Note that $G'$ is
finite so $\mathbb{X}$ is indeed Deligne-Mumford. We first show
that $\bar{U}_t$ meets every $G$-orbit. Indeed, if $x \in U_t$,
then $t(x) = \beta \in k$ is non-zero and there is some $g \in G$
such that $t(g.x) = 1$, just pick $g$ so that $\gamma (g) =
\beta^{-1}$. Let now $x,y \in \bar{U}_t$ lie in the same $G$-orbit
so say $y = g.x$. We need to show that $g\in G'$. If this is not
the case then $\gamma(g) \neq 1$ so $t(y) = \gamma(g) t(x) \neq 1$
so $y \notin \bar{U}_t$, a contradiction. This completes the
heuristics. 

Finally, we elucidate the induced category equivalence $\qcoh
[U_t/G] \simeq \qcoh [\bar{U}_t/G']$ which amounts to a category
equivalence $\Gr R[t^{-1}] \simeq \Gr R/(t-1)$ where $R[t^{-1}]$
is $\Gamma$-graded in the obvious way and $R/(t-1)$ is
$\bar{\Gamma}$-graded as in \cite[p.104]{Sm}. We now generalise a theorem of Smith-Zhang which can be found in \cite[Proposition~2.4]{Sm}. Note that there is a surjective ring homomorphism $\phi:R[t^{-1}] \lm R/(t-1)$. 

\begin{prop} \label{pstackred} 
The  natural morphism  $\iota:[\bar{U}_t/G'] \lm [U_t/G]$ induces the category equivalence $\Gr R[t^{-1}] \simeq
	\Gr R/(t-1)$ given by $\iota^* = R/(t-1) \otimes_{R[t^{-1}]}(-)$.
\end{prop} 
\begin{proof} The induced functor $\iota^*$ is given by $R/(t-1) \otimes_{R[t^{-1}]}(-)$  since it comes from the closed imbedding $\bar{U}_t \lm U_t$. 
We wish to define the inverse equivalence $\Phi$ so consider $M \in \Gr R/(t-1)$. We need to convert the additive notation of $\Gamma$ to multiplicative, so introduce the ``placeholder'' notation $s^{\delta}, \delta \in \Gamma$ with the understanding that $s^{\delta}s^{\eps} = s^{\delta + \eps}, \delta,\eps \in \Gamma$. We
	define the ``unwrap'' module $M[s] \in \Gr R[t^{-1}]$  by 
$$ M[s]_{\delta} = M_{\bar{\delta}} s^{\delta}  $$
where $\bar{\delta}$ denote the image of $\delta$ in $\bar{\Gamma}$. Given homogeneous elements $a \in
	R[t^{-1}]_{\eps}, ms^{\delta} \in M[s]_{\delta}$, we define
	$ams^{\delta} = \phi(a) m s^{\eps + \delta}$. This construction is clearly functorial and gives the sought for inverse equivalence $\Phi$ to $\iota^*$. 
\end{proof}

\textbf{Example} Let $\Gamma = (\Z \gamma_0 + \Z \gamma_1 + \Z \gamma_2)/(p_0\gamma_0 = p_1\gamma_1 = p_2\gamma_2) $ and $R = k[x_0,x_1,x_2]/(x_2^{p_2} + \lambda x_0^{p_0} -  x_1^{p_1})$ be the $\Gamma$-graded algebra with $\deg x_i = \gamma_i$. According to \cite{GL}, this gives the weighted projective line, weighted at at most 3 points $0, \infty, \lambda$ with weights $p_0, p_1, p_2$. Now $\mathbb{X} = [\left( \spec R - \text{pt}\right)/G]$ is covered by the open patches $x_0 \neq 0, x_1 \neq 0$. Write $t_0 = x_0^{p_0}, t_1 = x_1^{p_1}$. Note that  we have $R[x_0^{-1}]_0 = k[t_1/t_0]$ and $R[x_1^{-1}]_0 = k[t_0/t_1]$ so patching the affine lines together gives $\mathbb{P}^1$ as the coarse moduli space. Let us examine the patch $x_0 \neq 0$ more closely. $R/(x_0 -1) \simeq k[t_1,x_1,x_2]/(x_1^{p_1} - t_1, x_2^{p_2} - (t_1 - \lambda))$ which is the $\mu_{p_1} \times \mu_{p_2}$-cover of $\mathbb{A}_{t_1}$ ramified at $t_1=0,t_1=\lambda$ with ramification indices $p_1,p_2$ respectively. Of course, $\mu_{p_1} \times \mu_{p_2} = (\bar{\Gamma})^{\vee} =  (\Gamma/\Z \gamma_0)^{\vee}$ so the open substack we get here is $\left[ \left(\spec \frac{R}{(x_0-1)}\right) / \mu_{p_1} \times \mu_{p_2}\right]$. This is just the orbifold stack on $\mathbb{A}^1_{t_1}$ with stacky points at $t_1 = 0,\lambda$ with inertia groups there of $\mu_{p_1},\mu_{p_2}$ respectively.

\vspace{5mm}


\begin{thebibliography}{10}
\bibitem[ACV]{ACV} D. Abramovich, A. Corti, A. Vistoli. {\em Twisted bundles and admissible covers} Communications in Algebra 31.8 (2003): 3547-3618.
\bibitem[ASS]{ASS} I. Assem, D. Simson, A. Skowroński \emph{Elements of the representation theory
	of associative algebras. Vol. 1.}, Techniques of representation theory. London Mathematical
	Society Student Texts, 65. Cambridge University Press, Cambridge, (2006).
\bibitem[ATV]{ATV} M. Artin, J. Tate,  M. Van den Bergh {\em Some algebras associated to automorphisms of elliptic curves} The Grothendieck Festschrift. Birkhäuser Boston, (2007) 33-85.
\bibitem[AU]{AU}T. Abdelgadir, K. Ueda, \emph{Weighted projective lines as fine moduli spaces of
	quiver representations}, Communications in Algebra 43.2 (2015): 636-649.
\bibitem[Be]{Be} A.~A.~Be{\u\i}linson, \emph{Coherent sheaves on {${\bf P}^{n}$} and problems in linear algebra}, Funktsional. Anal. i Prilozhen. 12 (1978), no.~3, 68--69.
\bibitem[Br]{Br} T.Bridgeland, \emph{Equivalences of triangulated categories and Fourier–Mukai transforms} Bulletin of the London Mathematical Society 31.1 (1999): 25-34.
\bibitem[BKR]{BKR} T. Bridgeland, A. King, M. Reid,  \emph{The McKay correspondence as an
	equivalence of derived categories}, J. Amer. Math. Soc. 14 (2001), no. 3, 535--554 (electronic).
\bibitem[BO]{BO} A. Bondal, D. Orlov, \emph{Reconstruction of a variety from the
	derived category and groups of autoequivalences}. Compositio Math. 125 (2001), no. 3, 327--344. 
\bibitem[CE]{CE} H. Cartan, S. Eilenberg, \emph{Homological algebra}.Princeton University Press, Princeton, N. J., (1956). xv+390 pp. 
\bibitem[E]{E} D. Eisenbud,  \emph{Commutative algebra. With a view toward algebraic geometry}. Graduate Texts in Mathematics, 150. Springer-Verlag, New York, (1995). xvi+785 pp. ISBN: 0-387-94268-8; 0-387-94269-6
\bibitem[GL]{GL} W. Geigle; H. Lenzing, \emph{A class of weighted projective curves arising in representation theory of finite-dimensional algebras. Singularities, representation of algebras, and vector bundles} (Lambrecht, 1985), Lecture Notes in Math., 1273, Springer, Berlin, (1987).
\bibitem[Hap]{H} D. Happel, \emph{ On the derived category of a finite-dimensional algebra}.  Comment. Math. Helv. 62 (1987), no. 3, 339--389.
\bibitem[Har]{Har} R. Hartshorne, \emph{Algebraic geometry}. Graduate Texts in Mathematics, No. 52. Springer-Verlag, New York-Heidelberg, (1977).
\bibitem[HIMO]{HIMO} M. Herschend, O. Iyama, H. Minamoto, S. Oppermann, {\em Representation theory of Geigle-Lenzing complete intersections} arXiv preprint arXiv:1409.0668 (2014).
\bibitem[HIO]{HIO} M. Herschend, O. Iyama, S. Oppermann,  \emph{$n$-representation infinite algebras}. Adv. Math. 252 (2014), 292--342. 
\bibitem[HM]{HM} J. Harris, I. Morrison {\em Moduli of curves} Vol. 187. Springer Science \& Business Media, (2006).
\bibitem[K1]{K1} A. D. King, \emph{ Moduli of representations of finite-dimensional algebras}. Quart.
	J. Math. Oxford Ser. (2) 45 (1994), no. 180, 515--530.
\bibitem[K2]{K2} A. King, {\em Tilting bundles on some rational surfaces} Unpublished manuscript (1997).
\bibitem[Kr]{Kr}  K. Behrend, B. Conrad, D. Edidin, B. Fantechi, W. Fulton, L.
G\"{o}ttsche, A. Kresch. {\em Algebraic stacks} available on Andrew Kresch's website  \url{http://www.math.uzh.ch/index.php?pr_vo_det&key1=1287&key2=580&no_cache=1}  (in progress)                                                                                                                                                                   
\bibitem[KM]{KM} J. Koll\'{a}r, S. Mori, S.. \emph{Birational geometry of algebraic varieties}, Vol. 134. Cambridge University Press, 2008.
\bibitem[L]{L} H. Lenzing, {\em Hereditary categories} Handbook of tilting theory, 105-146. London Math. Soc. Lecture Note Ser 332.
\bibitem[LdP]{LdP} H. Lenzing, J. A.  De La Pe\~{n}a, {\em Concealed-canonical algebras and separating tubular families.} Proceedings of the London Mathematical Society 78.03 (1999): 513-540.
\bibitem[LM]{LM} H. Lenzing, M. Hagen,  {\em Tilting sheaves and concealed-canonical algebras} Representation theory of algebras Vol. 18. Amer. Math. Soc. Providence, RI, (1996). 455-473.
\bibitem[LM-B]{LM-B} G. Laumon, L.  Moret-Bailly {\em  Champs alg\'{e}briques} Vol. 39. Springer Science \& Business Media, (2000).
\bibitem[Mil]{Mil} J. S. Milne, {\em Etale Cohomology} (PMS-33). No. 33. Princeton university press, (1980).
\bibitem[R]{R} C. M. Ringel, \emph{Tame algebras and integral quadratic forms}, Lecture Notes in Mathematics, 1099. Springer-Verlag, Berlin, (1984).
\bibitem[S]{S} N. Spaltenstein, \emph{Resolutions of unbounded complexes}. Compositio Math. 65 (1988), no. 2, 121--154.  Vol. 1.
\bibitem[Sm]{Sm} S. P. Smith, {\em Computation of the Grothendieck and Picard groups of a toric DM stack $\mathcal{X}$ by using a homogeneous coordinate ring for $\mathcal{X}$} Glasgow Mathematical Journal 53.01 (2011): 97-113.
\bibitem[Stacks]{Stacks} The {Stacks Project Authors}, \emph{Stacks Project}. \url{http://stacks.math.columbia.edu} (2014) 
\bibitem[V]{V} A. Vistoli, {\em Intersection theory on algebraic stacks and on their moduli spaces} Inventiones mathematicae 97.3 (1989): 613-670.
\bibitem[W]{W} C. A. Weibel, {\em An introduction to homological algebra} No. 38. Cambridge university press (1995)
\end{thebibliography}

\end{document}